\documentclass[onefignum,onetabnum]{siamart171218}

\usepackage{lipsum}
\usepackage{amsfonts}
\usepackage{graphicx}
\usepackage{epstopdf}
\usepackage{algorithmic}
\ifpdf
  \DeclareGraphicsExtensions{.eps,.pdf,.png,.jpg}
\else
  \DeclareGraphicsExtensions{.eps}
\fi

\newsiamremark{remark}{Remark}
\newsiamremark{hypothesis}{Hypothesis}
\crefname{hypothesis}{Hypothesis}{Hypotheses}
\newsiamthm{claim}{Claim}

\headers{An equation modeling a camphor motion}{J.~Fan, M.~ Nagayama, G.~Nakamura and M.~Uesaka}

\title{Short-time approximate solutions of an equation modeling a camphor motion}

\author{
    Jishan Fan%
    \thanks{Department of Applied Mathematics, Nanjing Forestry University, Nanjing 210037, China (\email{fanjishan@njfu.edu.cn}).}
    \and Masaharu Nagayama%
    \thanks{Research Institute for Electronic Science, Hokkaido University, N12W7, Kita-Ward, Sapporo, Hokkaido, 060-0812, Japan.
  (\email{nagayama@es.hokudai.ac.jp}, \email{muesaka@es.hokudai.ac.jp}).}
    \and Gen Nakamura%
    \thanks{Department of Mathematics, Hokkaido University, N10W8, Kita-Ward, Sapporo, Hokkaido, 060-0810, Japan.
  (\email{gennakamura@gmail.com}).}
    \and Masaaki Uesaka\footnotemark[2]
}

\usepackage{amsopn}

\makeatletter
\newcommand*{\addFileDependency}[1]{
  \typeout{(#1)}
  \@addtofilelist{#1}
  \IfFileExists{#1}{}{\typeout{No file #1.}}
}
\makeatother

\newcommand*{\myexternaldocument}[1]{%
    \externaldocument{#1}%
    \addFileDependency{#1.tex}%
    \addFileDependency{#1.aux}%
}

\ifpdf
\hypersetup{
  pdftitle={An equation modeling a camphor motion},
  pdfauthor={M.~Nagayama, G.~Nakamura and M.~Uesaka}
}
\fi

\myexternaldocument{ex_supplement}

\newcommand{\xc}{x_\mathrm{c}}
\newcommand{\xic}{\xi_\mathrm{c}}
\newcommand{\tilxic}{\widetilde{\xi}_\mathrm{c}}
\newcommand{\rc}{r_\mathrm{c}}

\newcommand{\bdry}[1]{\partial#1}
\newcommand{\id}{\,\mathrm{d}}
\newcommand{\sid}[1]{\,\mathrm{d}\sigma_{#1}}
\newcommand{\pdnormal}{\partial_\nu}
\newcommand{\RR}{\mathbb{R}}
\newcommand{\abs}[1]{\left|#1\right|}

\newcommand{\tilJ}{\widetilde{J}}
\newcommand{\tilU}{\widetilde{U}}
\newcommand{\tilu}{\widetilde{u}}
\newcommand{\tilxc}{\widetilde{x}_{\mathrm{c}}}
\newcommand{\yc}{y_{\mathrm{c}}}

\newcommand{\pd}[1]{\partial_{#1}}
\newcommand{\pmt}[1]{%
    \begin{pmatrix}
        #1
    \end{pmatrix}
}
\newcommand{\diffop}[3]{
\ifnum #1=1
    \frac{\mathrm{d}#2}{\mathrm{d}#3}%
\else%
    \frac{\mathrm{d}^{#2}#1}{\mathrm{d}#3^{#1}}%
\fi
}%

\newcommand{\calZ}{\mathcal{Z}}
\newcommand{\norm}[2]{\left\|#1\right\|_{#2}}
\newcommand{\ZMT}{\calZ_{M,T}}

\newcommand{\Utzero}{U^{t_0}}
\newcommand{\Vtzero}{V^{t_0}}
\newcommand{\omegacap}[1]{(#1+\omega)\cap \Omega}
\begin{document}

\maketitle

\begin{abstract}
As a profound example of spontaneous motion,
we analyze the motion of a camphor particle on a water surface.
The motion is modeled as an initial boundary value problem for a coupled nonlinear system of a diffusion equation and an ordinary differential equation in a two-dimensional domain.
Since it seems that the well-posedness of this initial boundary value problem is missing,
we provided its proof.
Then, by constructing an approximate solution to this initial boundary value problem,
we gave a mathematically rigorous interpretation of a camphor motion.
That is we showed that the motion of a camphor locally in time has a self-avoiding orbit.
We also gave the numerical performance of the approximate solution. 
\end{abstract}


\section{Introduction}${}$\\

Spontaneous motions appear in several fields including biology, chemistry and nonlinear physics.
For example,
molecular motors in living organisms~\cite{Mallik2004},
bacteria swarming~\cite{Kaiser2007},
self-propelled motion of a catalystic nanoparticle~\cite{Paxton2005}, 
motion of a surfactant particle at water surface~\cite{SatoshiNakata2003,SatoshiNakata1998,SatoshiNakata1997},
droplet motion~\cite{Magome1996,Nagai2005,Sumino2005,Tersoff2009} are known as spontaneous motions.

Among these, a simple but profound example of self-propelled motion is that of camphor particles on water surface.
The study of this motion was
originated by the report~\cite{Rayleigh1889} by Rayleigh in 1889.
Since then, many theoretical and experimental studies have revealed the mechanism of the spontaneous motion of camphor particles. 
According to these studies (\cite{Hayashima2001,Nagayama2004a,SatoshiNakata1997}, for example), the mechanism of this motion is explained as follows:
When a camphor particle is set afloat on the water surface, 
the camphor dissolves in the water.
The camphor acts as surfactant and decreases the surface tension of the water
and this produces the spatial difference of the surface tension.
This difference of the surface tension is the driving force of the self-propelled motion.

The mathematical modeling and its analysis of camphor particle motions have been widely studied.
The widely-used model is a coupled system of a reaction-diffusion equation describing the camphor concentration on the water surface and an ordinary differential equation describing the motion of the camphor,
which is similar to the one given below as an initial boundary value problem \eqref{eq:diffusion}--\eqref{eq:initial}.
In one dimensional setting like the camphor motion in the thin water channel,
the detail mathematical studies and comparison with experiments have been developed, for examples, the oscillatory and unidirectional motion of one camphor disk on annular channel (\cite{Hayashima2001,Nagayama2004a}),
the synchronized motion of two camphor boats (\cite{Kohira2001})
and the motion of many camphor disks or boats like a traffic jam (\cite{Suematsu2010b}) have been reproduced numerically.
The bifurcation phenomena of the motion of two camphor disks have been analyzed in~\cite{Nishi2015}.
Moreover, in~\cite{Ei2014,Ei2015,Ikeda}, the rigorous analysis have been performed for the collective motion of the camphor boats by using center manifold theory and the detail motion of camphor boats is revealed.

A two dimensional model for a self-propelled motion by a surface tension appears in~\cite{Kitahata2005,Mikhailov:2006:CS:1201067}.
In the viewpoint of mathematical analysis, it was shown in~\cite{Mimura2009} that the radially symmetric solution giving a standing camphor particle is unstable under a small perturbation of the camphor when the camphor is circle-shaped.
The rotating motion of two camphor particles with fixed distance and center of mass was analyzed in~\cite{Koyano2017}.
Theoretical studies have been developed also for the non-symmetric camphor particles.
For example, by both mathematical analysis and experiments, it is revealed that an elliptic camphor particle is easier to move along short axis (\cite{Iida2014}) and that two elliptic particles interact so as to be parallel with long axes (\cite{Ei2018}).

In this article, based on~\cite{Mimura2009}, we consider the following mathematical model of the motion of a single circle-shaped camphor particle.
To begin with, 
let $\Omega$ be a domain in $\RR^2$ with $C^3$ boundary
$\partial\Omega$ and $\omega :=B_{\rc}(0) := \left\{ x \in \RR^2 \,;\, \abs{x} < \rc \right\}$ an open disk in $\RR^2$, which represents the shape of the camphor particle.
Assume that $x_0\in\Omega$ and the camphor is initially given as $x_0+\omega:=\{x_0+x\,;\, x\in\omega\}\subset\overline\Omega$.
Further we assume the camphor shape does not change during its motion. 
Then, the model is given as the initial boundary value problem:
\begin{align}
	& \pd{t}u = D \triangle u - \alpha u + f_{\rc}(x-\xc)\big|_{Q_T} & & \mbox{in}\ Q_T := (0,T) \times \Omega, \label{eq:diffusion} \\
    & \pdnormal u = 0 & & \mbox{on}\ (0,T) \times \bdry{\Omega}, \label{eq:Neumann}   \\
    & u(0,\cdot) = u_0 & & \mbox{in}\ \Omega,
    \label{eq:u-initial}   \\
	& \mu \diffop{1}{}{t}\xc = \int_{\partial\big((\xc+\omega)\cap\Omega\big)} \gamma(u(\cdot,y))\nu_y \sid{y} & & \mbox{in}\ (0,T), \label{eq:ODE} \\
	& \xc(0) = x_0. \label{eq:initial}
\end{align}
 Here ``$\big|_{Q_T}$'' denotes the restriction to $Q_T$,
 $\partial_\nu$ is the outer normal derivative at $\partial\Omega$,
 $\nu_y$ is the outer unit normal of the boundary $\partial\big((\xc+\omega)\cap\Omega\big)$ of $(\xc+\omega)\cap\Omega$
 and $\mathrm{d}\sigma_y$ is the line element.
 Note that since $(\xc+\omega) \cap \Omega$ has a Lipschitz boundary,
 $\nu_y$ is defined almost everywhere.
This initial boundary value problem is a coupled nonlinear system of a parabolic equation~\eqref{eq:diffusion} and an ordinary differential equation~\eqref{eq:ODE}.

The physical meaning of the above model and the assumptions which we are going to put to the model are as follows.
Equation~\eqref{eq:diffusion} describes the diffusion, the sublimation and supply from the camphor particle on the water surface.
Here the function $u$ represents the surface concentration of camphor.
The positive constants $D$ and $\alpha$ are the diffusion and sublimation coefficients, respectively.
Also, the function $f_{\rc}$ represents the camphor source from the particle
and it is usually given in the form of a characteristic function: 
    \begin{equation}\label{f_rc}
    f_{\rc} = F_0\mathbf{1}_\omega
\end{equation}
with the characteristic function $\mathbf{1}_\omega$ of $\omega$,
where the constant $F_0 > 0$ denotes the supply rate of camphor.
The viscosity coefficient $\mu$ is a positive constant and the surface tension $\gamma \in C^2([0,\infty))$ is positive and monotone decreasing.
We further assume for $\gamma$ that
\begin{equation}\label{beta}
\beta_j := \sup_{s \in \RR} \abs{\gamma^{(j)}(s)}<\infty,\ j=0,1,2.
\end{equation}
Further, as for the data $u_0$ and $x_0$ we assume that
\begin{equation}
\begin{array}{ll}
u_0\in W^{1,\infty}(\Omega),\ u_0\ge 0\ 
\text{in}\ \overline\Omega\\
\text{with the compatibility condition}\ \partial_\nu u=0\ \text{on}\ \partial\Omega
\end{array}
\end{equation}
and
\begin{equation}
x_0\in\overline\Omega,
\end{equation}
respectively. Here $W^{m,\infty}(\Omega)$ denotes the $L^\infty$-Sobolev space of order $m\in \mathbb{Z}_+:=\mathbb{N}\cup\{0\}$ in $\Omega$.

Our model is just a small modification of the original model appeared in~\cite{Mimura2009}.
More precisely, we modified the camphor source $f_{\rc}(x-\xc)$ to $f_{\rc}(x-\xc)\big|_{Q_T}$ and the domain of the boundary integral in~\eqref{eq:ODE} from $\partial(\xc+\omega)$ in the original model to $\partial\big((\xc+\omega) \cap \Omega)\big)$.
This modification is due to make the equation~\eqref{eq:ODE} mathematically meaningful
when the camphor particle $\xc+\omega$ touches the boundary $\partial \Omega$ of the domain
and even move more to $\partial\Omega$.
As a consequence if $(\xc+\omega)\cap\Omega=\emptyset$,
then $\xc$ stops moving.
This means that for small enough $\omega$, $\xc$ can represent the camphor particle $\xc+\omega$
and it will be trapped at the boundary $\partial\Omega$.
So far there is not any good model which can give the behavior of camphor particle movement at the boundary. 
It should be remarked here that in our model,
the particle is considered to be very light so that the inertial force can be neglected,
and hence the above first order equation~\eqref{eq:ODE} is chosen to describe the motion of a single camphor.


We further remark that this model is a special case of a more comprehensive model in~\cite{Ei2018,Iida2014,Nakata2018}
which additionally has a moment equation of the camphor particle coming from the lack of radial symmetry.
Nevertheless, this model is still useful
because it can be used to interpret some experimental observations of a (non-symmetric) camphor particle.

In spite of the importance of the initial boundary value problem \eqref{eq:diffusion}--\eqref{eq:initial} for analyzing the self-propelled motion,
it seems that there is no result on its well-posedness as long as the authors know.
Our first result in this paper is the following well-posedness for this initial boundary value problem.
Here and henceforth, we assume without losing generality that $x_0=0$, because we can easily modify the result by translation when $x_0 \neq 0$.

\begin{theorem}\label{thm:well-posedness}
    Let $T > 0$ and $M > 0$.
    Then there exists $r^\ast = r^\ast(M,T) > 0$ such that for any $0<\rc < r^\ast$, 
    the initial boundary value problem \eqref{eq:diffusion}--\eqref{eq:initial} has a unique solution $(u,\xc)$ with the regularity:
    \begin{equation}\label{regularity}
        u \in C^0([0,T];W^{1,\infty}(\Omega))\cap C^1((0,T];L^p(\Omega))\ \mbox{\rm for any $p\ge 2$ and}\  
        \xc \in C^1([0,T]); 
    \end{equation}
    the non-negativity:
    \begin{equation}\label{non-negativity}
    u(t,x)\ge 0,\,(t,x)\in [0,T]\times\overline\Omega
    \end{equation}
    and the estimate:
    \begin{equation}\label{bound}
    \norm{{\xc}}{1/2} \le M.
    \end{equation}
    Here  $\norm{\cdot}{1/2}$ is the norm of H\"older space $C^{1/2}([0,T])$ with H\"older exponent $1/2$ (see \eqref{Holder norm}).
    \end{theorem}

\begin{remark}
${}$
\begin{itemize}
    \item \cref{thm:well-posedness} states that for any time interval, 
    the solution of the initial boundary value problem system~\eqref{eq:diffusion}--\eqref{eq:initial} exists uniquely
    provided that the radius $\rc$ of the camphor particle is sufficiently small.
    We obtain this kind of unique solvability for arbitrarily fixed time $T>0$
    as long as the radius $\rc$ of the camphor particle is sufficiently small.
    The reason for this can be explained as follows:
    It follows immediately  from~\eqref{eq:ODE} that the a priori estimate
    \begin{equation}\label{speed}
        \norm{\diffop{1}{\xc}{t}}{C^0([0,T])} \le \frac{2\pi\rc\beta_0}{\mu}
    \end{equation}
    holds for any $T > 0$.
    This inequality means that the camphor particle moves more slowly as $\rc$ becomes smaller.
    Hence, even for small $M>0$ or large $T>0$, by choosing sufficiently small $\rc$, the camphor particle stays moving so that estimate \eqref{bound} holds.
    We also note that this type of well-posedness may fail when the inertial force term is contained in \eqref{eq:ODE} with assumption that the mass of the camphor particle depends on $\rc$.
    \item In~\cref{thm:well-posedness}, we took sufficiently small $\rc>0$ for given $T > 0$.
    Since \eqref{speed} tells us that the roles of making $\rc$ small and $T$ small are the same, we can also have the well-posedness for sufficiently small time $T > 0$ for arbitrarily fixed $\rc > 0$.
\end{itemize}
\end{remark}

\medskip
Main interest of the model~\eqref{eq:diffusion}--\eqref{eq:initial} is the orbit of $\xc(t)$,
which represents the motion of the camphor particle.
We investigate the detailed motion by constructing a short-time approximate solution of \eqref{eq:diffusion}--\eqref{eq:initial} with small $\rc$.
In order to achieve this, we follow the method in~\cite{Constantinescu2010,Nakagawa2014}.
This method is summarized as follows:
\begin{enumerate}
    \item Introduce new spatial and time variables with small parameter $\rc > 0$;
    \item Set an ansatz that the asymptotic solution of the scaled system is written as a formal power series in $\rc$ and derive a linear equation for the each term of the series.
\end{enumerate}
This method was the key to construct the fundamental solution of linear parabolic equation (\cite{Cheng2011a,Constantinescu2010})
and was utilized to obtain the closed-form asymptotics of the Black--Scholes equation (\cite{Cheng2011}).
This method was also applied to analyze the behaviour of a hot spot of a solution of a reaction-diffusion system which models the iron ore sintering process (\cite{Nakagawa2014}).

\medskip
Applying this method to the initial boundary value problem \eqref{eq:diffusion}--\eqref{eq:initial},
we find a {\it short time approximate solution} and its error estimate,
which is an another result of this paper.

\begin{theorem}\label{thm:main}
    Let $t_0\in (0,T)$ and a pair $(u,\xc)$ be the solution to the initial boundary value problem \eqref{eq:diffusion}--\eqref{eq:initial}.
    Assume that $\overline{\xc(t_0)+\omega}\subset\Omega$ and \begin{equation}
        \gamma(w)=\gamma_0-\gamma_1w+O(|w|^2),\,\,|w|\ll1
    \end{equation}
    with positive constants $\gamma_0,\,\gamma_1$. Define the function $\tilu(x,t)$ by
    \begin{equation}\label{eq:tilu}
    \begin{split}
        \tilu(t,x) &:= \rc^2u^{t_0}(x) + \rc^2(t - t_0)A_1 u^{t_0}(x) + \rc^2(t - t_0)^2 A_2 u^{t_0}(x) \\
        &\phantom{=} + \int_{t_0}^{t} \left[ f_{\rc}\left(x - \xc(t')\right) + (t - t')A_1f_{\rc}\left(x- \xc(t')\right) \right] \id t' \\
        &\phantom{=} + u_0(x) - \alpha u_0(x)(t - t_0) + \frac{1}{2} \alpha^2 u_0(x) (t - t_0)^2,\,\,t\sim t_0,\,\,x\in \xc(t_0)+\omega
    \end{split}
    \end{equation}
    with $u^{t_0}(x):= u(t_0,x)$.
    Further let $\tilde{x}_c(t)$ be the solution to following nonlinear integro-differential equation:
    \begin{equation}\label{eq:tilxc}
    \begin{split}
        \mu \diffop{1}{\tilxc}{t}(t) &= %
        - \gamma_1 \rc \left( v^{t_0}(\tilxc(t)) + (t - t_0)A_1v^{t_0}(\tilxc(t)) + (t - t_0)^2 A_2 v^{t_0}(\tilxc(t))\right) \\
        &\phantom{=}%
        - \gamma_1 \rc \left( \int_{t_0}^{t} \left[ \varphi\left(\tilxc(t) - \tilxc(t')\right) + (t - t')A_1\varphi\left(\tilxc(t) - \tilxc(t')\right) \right] \id t' \right),
    \end{split}
    \end{equation}
    for $t\sim t_0$ and $x\in \xc(t_0)+\omega$ with the initial condition $\tilxc(t_0) = \xc(t_0)$,
    where
    \begin{equation}
    \begin{split}
        v^{t_0}(x) &:= \int_{\bdry{B_{\rc}(0)}} u^{t_0}(x + y) \nu_y \id\sigma_y, \\
        \varphi(x) &:= \int_{\bdry{B_{\rc}(0)}} f_{\rc} \left(x + y\right)\nu_y \id\sigma_y, \\
        A_1 &:= D\triangle_x - \alpha, \\
        A_2 &:= \frac{1}{2}D^2\triangle_x^2 - \alpha \triangle_x.
    \end{split}
    \end{equation}
    Then for arbitrarily fixed $B > 0$,
    there exists $\delta > 0$ such that
    \begin{equation}
        \begin{array}{ll}
            \norm{u-\tilu}{C^0([t_0,t_0+\delta];W^{1,\infty}(\Omega))}\le B\rc^2, \\
            \norm{\xc - \tilxc}{C^0([t_0,t_0+\delta])} \le B\rc.
        \end{array}
    \end{equation}
    We call $(\tilde{u},\tilde{x}_c)$ the short time approximate solution of the initial boundary value problem \eqref{eq:diffusion}--\eqref{eq:initial}.
\end{theorem}

\begin{remark}
\cref{thm:main} tells us that if the time interval $[t_0,t_0+\delta]$ is sufficiently short and $\rc$ is sufficiently small,
the solution $(u,\xc)$ of the initial boundary value problem \eqref{eq:diffusion}--\eqref{eq:initial} is approximated by $(\tilu,\tilxc)$ defined by~\eqref{eq:tilu} and~\eqref{eq:tilxc}.
\end{remark}

\medskip
In order to see the physical meaning of \eqref{eq:tilxc},
we decompose its right hand side into the sum of the following three parts $\mathcal{F}_j\ (j=1,2,3)$:
\begin{align*}
    \mathcal{F}_1 &:= - \gamma_1 \rc \left[ v^{t_0}(\tilxc(t)) + (t - t_0)A_1v^{t_0}(\tilxc(t)) + (t - t_0)^2 A_2 v^{t_0}(\tilxc(t))\right], \\
    \mathcal{F}_2 &:= - \gamma_1 \rc \int_{t_0}^{t} \varphi\left(\tilxc(t) - \tilxc(t')\right)  \id t' , \\
    \mathcal{F}_3 &:= - \gamma_1 \rc \int_{t_0}^{t}  (t - t') A_1\varphi\left(\tilxc(t) - \tilxc(t')\right)\id t'.
\end{align*}
The term $\mathcal{F}_1$ represents the force exerted by the profile $u^{t_0}$ of the camphor concentration at time $t=t_0$. 
In order to explain the second term $\mathcal{F}_2$, we consider the direction of the vector $\varphi$.
By the radial symmetry of $f_{\rc}$, we can assume that $x$ is of form $\pmt{x_1 \\ 0}$.
By the change of variable $y = \rc \pmt{\cos \theta \\ \sin \theta}$, we obtain that
\begin{equation}
    \varphi(x) = \int_0^{2\pi} f_{\rc}(x_1 + \rc \cos \theta, \rc \sin \theta) \pmt{\cos \theta \\ \sin \theta} \id \theta.
\end{equation}
Since $f_{\rc} = F_0\mathbf{1}_\omega$ is radially symmetric, the second component of $\varphi(x)$ vanishes.
The first component is given by
\begin{equation}
    \int_0^{2\pi} f_{\rc}(x_1 + \rc \cos \theta, \rc \sin \theta)\id \theta = \int_{\theta_0}^{2\pi-\theta_0} \cos \theta \id \theta = - 2\sin \theta_0,
\end{equation}
where $0 < \theta_0 < \pi$ satisfies $\cos \theta_0 = \dfrac{x_1}{2\rc}$,
which comes from the condition that $\abs{\pmt{x_1 + \rc \cos \theta \\ \rc \sin \theta}} \le \rc$.
Hence the first component of the vector $\varphi(x)$ is non-positive and this implies that $\varphi(x)$ and $x$ are in opposite directions if $\varphi(x)$ itself does not vanish.
Then the vector $-\varphi(\tilxc(t) - \tilxc(t'))$ directs from $\tilxc(t')$ to $\tilxc(t)$.
Therefore we regard $\mathcal{F}_2$ as ``the repulsive force'' acting at $\tilxc(t)$ from the past orbit of the camphor particle itself.

Another integral term $\mathcal{F}_3$ is also related to the path of the camphor particle but the direction of the force is not definite.
If $t-t'$ is small, however, we expect that the contribution of $\mathcal{F}_3$ is much smaller than that of $\mathcal{F}_2$.
Thus we see that the surface tension around the path decreases and that the particle moves as if it avoided the previous path itself.
We would like to emphasize that this description of camphor motion in short time was derived rigorously by asymptotic analysis.
We will also check the performance of our approximate solution by numerical simulations.

\medskip
The rest of this paper is organized as follows.
In \cref{sec:well-posedness}, we prove \cref{thm:well-posedness}.
In \cref{sec:asymptotics}, we give an exposition of the derivation of the short time approximate solution $(\tilu,\tilxc)$.
In \cref{sec:error},
we complete the proof of \cref{thm:main} by showing the error estimate for the short time approximate solution.
Finally in \cref{sec:numerical},
we show the numerical results to show the error of the approximate solution.

\medskip
\section{Proof of \cref{thm:well-posedness}}\label{sec:well-posedness}${}$
\par
In this section, we prove \cref{thm:well-posedness}. The proof will be given as follows.
We first prove an existence of solution $(u(t,x), \xc(t))$ to the initial boundary value problem \eqref{eq:diffusion}--\eqref{eq:initial}.
Next we show its uniqueness.
Finally we will show the non-negativity of $u(t,x)$. 

Let start to prove an existence of solution $(u(t,x),\xc(t))$.
For fixed $M>0, T>0$, define
\[
\ZMT := \left\{z\in C^{1/2}([0,T])\,;\, \norm{z(\cdot)-x_0}{1/2} \le M \right\},
\]
where $C^{1/2}([0,T])$ is the set of H\"older continuous function of order $1/2$ defined on $[0,T]$ equipped with norm 
\begin{equation}\label{Holder norm}\norm{ z}{1/2}=\norm{z}{C^0([0,T])}+
\sup_{0\le t_1<t_2\le T}\frac{\abs{z(t_2)-z(t_1)}}{(t_2-t_1)^{1/2}}.
\end{equation}

By taking $\xc(t)=z(t)\in \ZMT$,
consider the initial boundary value problem \eqref{eq:diffusion}--\eqref{eq:u-initial} for $u=u(t,x;z)$,
and define the mapping $\Phi$ on $\ZMT$ by
\begin{equation}\label{eq:dfn_of_Phi}
    \Phi(z)(t) := x_0 + \frac{1}{\mu} \int_0^t \left( \int_{\bdry{((z(\tau)+\omega)\cap\Omega)}} \gamma(u(\tau, y;z))\nu_y \sid{y} \right) \id \tau
\end{equation}
for $z \in \ZMT$.
We will prove that $\Phi$ is a contraction mapping on
$\ZMT$ with respect to the $C^{1/2}([0,T])$ topology 
so that its fixed point satisfies \eqref{eq:diffusion}--\eqref{eq:initial}.
More precisely we will show the following three assertions (1)--(3).
\begin{enumerate}
    \renewcommand{\labelenumi}{(\arabic{enumi})}
    \item $\Phi$ maps $\ZMT$ into $\ZMT$.
    \item $\Phi$ is a contraction in $\ZMT$, that is, there exists a positive constant $c < 1$ such that
    \begin{displaymath}
        \norm{\Phi (z_1) - \Phi (z_2)}{1/2}\le c\norm{ z_1-z_2}{1/2}.
    \end{displaymath}
    for any $z_1,z_2 \in \ZMT$.
    We note that
    this also guarantees the uniqueness of solution to
    the initial value problem \eqref{eq:u-initial}--\eqref{eq:initial} for $z=\xc$.
    \item The unique fixed point of $\Phi$ in $\ZMT$ and satisfies the original system~\eqref{eq:diffusion}--\eqref{eq:initial}.
\end{enumerate}

\medskip
To show these assertions, we first prepare two preliminary facts.
The first one is about the following well known result on the unique solvability of the initial boundary value problem \eqref{eq:diffusion}--\eqref{eq:u-initial} with a general source term.

\begin{lemma}{(\cite{Ito1992})}\label{lemma:FS}
    Let $2\le p\le\infty$, $0<\lambda<1$, $g \in C^\lambda([0,T];L^p(\Omega))$ and $u_0\in W^{1,\infty}(\Omega)$,
    where  $C^\lambda([0,T];L^p(\Omega))$ denotes the set of $L^p(\Omega)$ valued H\"older continuous functions on $[0,T]$ with H\"older exponent $\lambda$.
    Assume that $u_0$ satisfies the compatibility condition $\partial_\nu u_0=0$ at $\partial\Omega$.
    Consider the following initial boundary value problem: 
    \[
        \begin{aligned}\label{solution u(g)}
	        & \pd{t}u - D \triangle u + \alpha u=g(t,x) & & \mbox{in}\ Q_T := (0,T) \times \Omega, \\
            & \pdnormal u = 0  & & \mbox{on}\ (0,T) \times \bdry{\Omega},  \\
            & u(0,\cdot) = u_0 & & \mbox{in}\ \Omega.
        \end{aligned}
    \]
    Then there exists a unique solution $u\in C^0([0,T]; W^{1,p}(\Omega))\cap C^1((0,T]; L^p(\Omega))$ of this initial boundary value problem and $u$ has the following representation:
    \begin{equation}\label{representation}
        u(t,x) =\int_\Omega \Gamma(t,x;0,y)u_0(y)\,\id y+\int_0^t \int_{\Omega} \Gamma(t,x;s,y) g(s,y) \id y \id s,
    \end{equation}
   where $\Gamma(t,x;s,y)$ is the Green function of the operator $\partial_t-D\Delta+\alpha$ in $Q_T$ with Neumann boundary condition at $(0,T)\times\partial\Omega$ and singularity at $(s,y)\in (0,T)\times\Omega$.
   More precisely $\Gamma(t,x;s,y)$ is a distribution defined in $\big((0,T)\times\Omega\big)\times\big((0,T)\times\Omega\big)$ such that it is $C^\infty$ in $(t,s)$ and $C^2$ in $(x,y)$ except $(t,x)=(s,y)$.
   Further, $\Gamma(s,x;s,y)=\delta(x-y)$, $\Gamma(t,x;s,y)=0$ for $t<s$ and it satisfies the following estimates:
    \begin{align}
        \abs{\Gamma(t,x;s,y)} &\le c_0 (t-s)^{-1} \exp\left(-c_2 \frac{\abs{x-y}^2}{t-s}\right), \label{eq:Gamma-0} \\
        \abs{\nabla_x \Gamma(t,x;s,y)} &\le c_1 (t-s)^{-3/2} \exp\left(-c_2 \frac{\abs{x-y}^2}{t-s}\right), \label{eq:Gamma-1}
    \end{align}
for $x,y \in \Omega$ and $t > s$,
where $c_0,c_1,c_2 > 0$ do not depend on $s$, $t$, $x$ and $y$.
\end{lemma}

\begin{remark}
 $\Gamma_0(t,x,y;s):=\Gamma(t,x;s,y)\big|_{t\ge s}$ is the fundamental solution of the Cauchy problem with Neumann boundary condition giving initial condition at $t=s$.
 Since the coefficients of our equation and the boundary operator of our boundary condition do not depend on time, $\Gamma_0(t,x,y;s)=\Gamma_0(t-s,x;0)$.
 The existence of the fundamental solution $\Gamma_0(t,x,y;s)$ and its properties can be seen in~\cite{Ito1992}.
 In terms of the fundamental solution,
 the Green function $\Gamma(t,x;s,y)$ is given as 
\[
\Gamma(t,x;s,y)=
\left\{
\begin{array}{ll}\Gamma_0(t,x,y;s)\,\,&\text{for}\,\,t\ge s,\\
0\,\,&\text{for}\,\,t<s.
\end{array}
\right.
\]
Having this in mind, it is quite standard to show that $u$ given by \eqref{representation} belongs to $C^0([0,T];W^{1,p}(\Omega))\cap C^1((0,T];L^p(\Omega))$
and is a solution to the initial boundary value problem.
As for the uniqueness,
we can just put it to the uniqueness of the $L^2$ theory (see~\cite{Wloka1987}).
\end{remark}

\medskip
The second preliminary fact is about the $L^p$-estimate for the difference of the characteristic functions.
\begin{lemma}\label{lemma:diff_char}
    Let $x_1,x_2 \in \RR^2$.
    Then for any $p\ge1$ we have
    \begin{equation}\label{ineq for char func}
        \left(\int_\Omega 
        \abs{\mathbf{1}_{\omegacap{x_1}} - \mathbf{1}_{\omegacap{x_2}}}^p \id y\right)^{1/p} \le \left(4 \rc \abs{x_1-x_2}\right)^{1/p}.
    \end{equation}
\end{lemma}
\begin{proof}
    Since the integrand of the left hand side of \eqref{ineq for char func} does not depend on $p\ge1$,
    it is enough to show \eqref{ineq for char func} for $p=1$.
    Let us denote $\ell := \abs{x_1 - x_2}$.
    The difference $\abs{\mathbf{1}_{\omegacap{x_1}} - \mathbf{1}_{\omegacap{x_2}}}$ is equal to the characteristic function of $((x_1+\omega) \ominus (x_2+\omega)) \cap \Omega$,
    where $A \ominus B := (B \setminus A) \cup (A \setminus B)$ is a symmetric difference of two sets $A,B$.
    Then we have
    \[
        \begin{aligned}
            &\int_\Omega 
        \abs{\mathbf{1}_{\omegacap{x_1}} - \mathbf{1}_{\omegacap{x_2}}} \id y \\
        & \le \int_\Omega 
        \mathbf{1}_{(x_1+\omega) \ominus (x_2 + \omega)}\id y \\
        & \le \begin{cases}
            4 \rc^2 \left(\dfrac{\pi}{2} - \arccos \dfrac{\ell}{2\rc} + \dfrac{\ell}{2\rc} \sqrt{1-\left(\dfrac{\ell}{2\rc}\right)^2}\right) & \text{\rm if}\,\,\ell < 2\rc \\
            2 \pi \rc^2 & \text{\rm if}\,\,\ell \ge 2\rc
        \end{cases}
        \end{aligned}
    \]
    by simple geometrical calculation.
    Since the function $f(\rho) = \dfrac{\pi}{2} - \arccos \rho + \rho \sqrt{1 - \rho^2}$ has the derivative $f'(\rho) = 2\sqrt{1-\rho^2}$,
    we have $f(\rho) \le \min \left\{2 \rho, \dfrac{\pi}{2} \right\}$ for any $0 \le \rho < 1$, which completes the proof.
\end{proof}

\medskip
The following lemma guarantees that the solution $u=u(t,x;z)$ of \eqref{solution u(g)} discussed in \cref{lemma:FS} with $g(t,x)=f_{\rc}(x-z(t))\big|_{Q_T}$ is available. Hence the map $\Phi(z)$ can be defined.

\begin{lemma} For $z\in \ZMT$, $g(t,x)=f_{\rc}(x-z(t))\big|_{Q_T}\in C^\lambda([0,T]; L^p(\Omega))$ with $\lambda=1/(2p)$.
\end{lemma}
\begin{proof}
    Let $t_1, t_2\in [0,T]$. Then we only need to observe the following.
    By Minkowski's inequality and \cref{lemma:diff_char}, we have
    \begin{equation}
    \begin{aligned}
        &\abs{\norm{ g(t_1,\cdot)}{L^p(\Omega)}-\norm{ g(t_2,\cdot)}{L^p(\Omega)}} \\
        &\le F_0\left(\int_\Omega| \mathbf{1}_\omega(x-z(t_1))-\mathbf{1}_\omega(x-z(t_2))|^p\id x\right)^{1/p} \\
        &\le F_0(4r_{\rc}\norm{ z}{1/2})^{1/p}\,\abs{t_1-t_2}^{1/(2p)},
    \end{aligned}
\end{equation}
which completes the proof.
\end{proof}

\medskip
Now we are ready to show the first assertion.
\begin{lemma}
    There exists $r^\ast = r^\ast(M,T) > 0$ such that
    if $\rc \le r^\ast$, 
    the mapping $\Phi$ defined by~\eqref{eq:dfn_of_Phi} maps $\ZMT$ into itself.
\end{lemma}
\begin{proof}
Since $\gamma$ is bounded, we can easily have
    \begin{displaymath}
        \norm{\Phi(z)-x_0}{C^0([0,T])} \le \frac{2\pi \beta_0 {\rc}}{\mu} T.
    \end{displaymath}
For $0\le t_1<t_2\le T$, denote the difference operator $\delta_{t_2,t_1}$ of
$\Phi(z)$ by
\begin{equation}\label{difference}
\delta_{t_2,t_1}\Phi(z)=\Phi(z)(t_2)-\Phi(z)(t_1).
\end{equation}
Then we have
\begin{equation*}
\begin{aligned}
\abs{\delta_{t_2,t_1}\Phi(z)} &= \abs{\mu^{-1}\int_{t_1}^{t_2}\left(\int_{\partial((z(\tau)+\omega)\cap\Omega)}\gamma(u(\tau,y;z))\nu_y\id \sigma_y\right)\id\tau} \\
&\le \frac{2\pi{\rc}\beta_0}{\mu}\abs{t_1-t_2} \\
&\le \frac{2\pi{\rc}\beta_0}{\mu}\sqrt{2T}\abs{t_1-t_2}^{1/2}.
\end{aligned}
\end{equation*}
Hence we have obtained
\begin{equation}\label{estimate of Phi(z)}
\norm{\Phi(z)-x_0}{C^1([0,T])}\le \frac{2\pi\beta_0{\rc}}{\mu}\left(T+\sqrt{2T}\right).
\end{equation}
Then we can complete the proof by just taking $\rc>0$ to satisfy
\begin{equation}
    \rc\le \frac{\mu M}{2\pi\beta_0}\left(T+\sqrt{2T}\right)^{-1}.
    \end{equation}
\end{proof}

\medskip
\begin{remark}
Note that we obtained the estimate~\eqref{estimate of Phi(z)} which does not depend on $u=u(t,x;z)$ thanks to the boundedness of $\gamma$.
\end{remark}

\medskip
Now we prove the contraction property of $\Phi$.
\begin{lemma}\label{lemma:contraction}
       There exist $r^\ast = r^\ast(M,T) > 0$ and $0 < c < 1$ such that
        \begin{displaymath}
        \norm{\Phi (z_1) - \Phi (z_2)}{1/2} \le c\norm{z_1-z_2}{1/2},\,\,z_1,z_2 \in \ZMT
    \end{displaymath}
    for any $0<\rc \le r^\ast$.
\end{lemma}
\begin{proof}
    For simplicity, put $u_j:= u(\cdot,\cdot;z_j)$ for $j=1,2$.
    Also, general positive constants in the forthcoming estimates which do not depend on $z_1, z_2\in\ZMT$ will be denoted by $C$.
    We first estimate
    the $C^0([0,T])$-norm of $\Phi(z_1)-\Phi(z_2)$.
    Observe that
    \begin{align*}
            &\Phi(z_1)(t) - \Phi(z_2)(t) \\%
            &=\frac{1}{\mu} \int_0^t \left( \int_{\bdry{(\omegacap{z_1(\tau)})}} \gamma(u_1(\tau, y)) \nu_y \sid{y} 
            - \int_{\bdry{(\omegacap{z_2(\tau)})}} \gamma(u_2(\tau, y)) \nu_y \sid{y} \right) \id \tau \\
            &= \frac{1}{\mu} \int_0^t \int_{\bdry{(\omegacap{z_1(\tau)})}} \left[ \gamma(u_1(\tau, y)) - \gamma(u_2(\tau, y)) \right]\nu_y \sid{y} \id \tau \\
            &\phantom{=} + \frac{1}{\mu} \int_0^t \left( \int_{\bdry{(\omegacap{z_1(\tau)})}} \gamma(u_2(\tau, y)) \nu_y \sid{y} 
            - \int_{\bdry{(\omegacap{z_2(\tau)}})} \gamma(u_2(\tau, y)) \nu_y \sid{y} \right) \id \tau \\
            &= \frac{1}{\mu} \int_0^t \int_{\omegacap{z_1(\tau)}} \left[ \gamma'(u_1(\tau, y)) \nabla u_1(\tau, y)- \gamma'(u_2(\tau, y)) \nabla u_2(\tau, y) \right] \id y \id \tau \\
            &\phantom{=} + \frac{1}{\mu} \int_0^t \left( \int_{\bdry{\omegacap{z_1(\tau)}}} \gamma(u_2(\tau, y)) \nu_y \sid{y} 
            - \int_{\bdry{\omegacap{z_2(\tau)}}} \gamma(u_2(\tau, y)) \nu_y \sid{y} \right) \id \tau.
        \end{align*}
    Let us divide $\Phi(z_1)(t) - \Phi(z_2)(t)$ into three parts:
    \begin{align*}
            &\Phi(z_1)(t) - \Phi(z_2)(t) = I_1 + I_2 + I_3, \\
            I_1 &:= \frac{1}{\mu} \int_0^t \int_{\omegacap{z_1(\tau)}} \left[\gamma'(u_1(\tau, y)) - \gamma'(u_2(\tau, y))\right]\nabla u_1(\tau,y) \id y \id \tau, \\
            I_2 &:= \frac{1}{\mu} \int_0^t \int_{\omegacap{z_1(\tau)}} \gamma'(u_2(\tau,y)) \left[\nabla u_1(\tau,y) - \nabla u_2(\tau,y) \right] \id y \id \tau, \\
            I_3 &:= \frac{1}{\mu} \int_0^t \left( \int_{\bdry{\omegacap{z_1(\tau)}}} \gamma(u_2(\tau, y)) \nu_y \sid{y} 
            - \int_{\bdry{\omegacap{z_2(\tau)}}} \gamma(u_2(\tau, y)) \nu_y \sid{y} \right) \id\tau.        
    \end{align*}
    By using the mean value theorem, we have
    \begin{align*}
       \abs{I_1}\le &\frac{\beta_2}{\mu} \int_0^t \int_{\omegacap{z_1(\tau)}} \abs{\tilu(\tau,y)}\abs{\nabla u_1(\tau,y)} \id y \id \tau\\
       \le &\frac{\beta_2}{\mu}\norm{ u_1}{C^0([0,T];W^{1,\infty}(\Omega))}\int_0^t\int_{(z_1(\tau)+\omega)\cap\Omega}|\tilde u(\tau,y)|\,dy\,d\tau\\
       \le & \frac{\beta_2 T}{\mu}\norm{ u_1}{C^0([0,T];W^{1,\infty}(\Omega))}\,\norm{\tilde u}{C^0([0,T];L^1(\Omega))},
       \end{align*}
    where $\tilu := u_1 - u_2$.
    Here, \cref{lemma:FS} and Young's inequality, we have
    $$
    \norm{ u_1}{C^0([0,T];W^{1,\infty}(\Omega))}\le C(\norm{ u_0}{L^\infty(\Omega)}+F_0T).
    $$
    Also, by using \cref{lemma:diff_char}, a similar argument gives the estimate 
    \[
        \norm{\tilu}{C^0([0,T];L^1(\Omega))} \le CF_0T\rc \norm{z_1 - z_2 }{C^0([0,T])}.
    \]
    Hence
    \begin{equation}\label{eq:estimate_I1}
        \abs{I_1} \le \frac{C\beta_2F_0T^2\rc}{\mu}  \left(\norm{u_0}{L^\infty(\Omega)} + F_0T\right)\norm{z_1 - z_2}{C^0([0,T])}.
    \end{equation}
    
    Next by an argument similar to $I_1$, we have
    \[
        \abs{I_2} \le \frac{\beta_1}{\mu} \int_0^{t} \int_{\omegacap{z_1(\tau)}} \abs{\nabla \tilu(\tau,y)}\id y \id \tau.
    \]
    From the estimate~\eqref{eq:Gamma-1} of \cref{lemma:FS}, we have
    \[
        \begin{aligned}
            \int_0^T \int_\Omega \abs{\nabla_x\Gamma(t,x;s,y)} \id y \id s \le C\sqrt{T} & \mbox{ uniformly in $(t,x)$}, \\
            \int_0^T \int_\Omega \abs{\nabla_x\Gamma(\tau,x;s,y)} \id x \id \tau \le C\sqrt{T} & \mbox{ uniformly in $(s,y)$}.
        \end{aligned}
    \]
    Then Young's inequality and \cref{lemma:diff_char} yield that
    \[
        \norm{\nabla\tilu}{L^1((0,t)\times \Omega)} \le CF_0\sqrt{T}\rc \norm{z_1 - z_2}{C^0([0,T])},
    \]
    and hence we obtain
    \begin{equation}\label{eq:estimate_I2}
        \abs{I_2} \le \frac{C\beta_1F_0\sqrt{T}\rc}{\mu}\norm{z_1 - z_2}{C^0([0,T])}.
    \end{equation}
    
    Finally we estimate $I_3$.
    By integration by parts, we have
    \[
        I_3 %
        = \frac{F_0}{\mu} \int_0^t \int_\Omega \left( \mathbf{1}_{\omegacap{z_1(\tau)}} - \mathbf{1}_{\omegacap{z_2(\tau)}} \right) \gamma'(u_2(\tau,y)) \nabla u_2(\tau,y) \id y \id \tau.
    \]
    Then, by using \cref{lemma:diff_char}, we have
    \begin{equation}\label{eq:estimate_I3}
        \begin{split}
            \abs{I_3} &\le \frac{F_0}{\mu} \int_0^t \int_\Omega \abs{ \mathbf{1}_{\omegacap{z_1(\tau)}} - \mathbf{1}_{\omegacap{z_2(\tau)}}} \abs{\gamma'(u_2(\tau,y))} \abs{\nabla u_2(\tau,y)} \id y \id \tau \\
            &\le \frac{\beta_1F_0}{\mu} \int_0^{t} \int_{\Omega} \norm{\nabla u_2}{C^0([0,T];L^\infty(\Omega))}\abs{ \mathbf{1}_{\omegacap{z_1(\tau)}} - \mathbf{1}_{\omegacap{z_2(\tau)}}} \id y \id \tau \\
            &\le \frac{C \beta_1F_0}{\mu} \int_0^{t} \int_{\Omega} \norm{\nabla u_2}{C^0([0,T];L^\infty(\Omega))} \abs{ \mathbf{1}_{\omegacap{z_1(\tau)}} - \mathbf{1}_{\omegacap{z_2(\tau)}}} \id y \id \tau \\
            &\le \frac{C \beta_1F_0 T \rc}{\mu}  \left(\norm{u_0}{L^\infty(\Omega)} + F_0T\right) \norm{z_1-z_2}{C^0([0,T])}.
        \end{split}
    \end{equation}
    
    \medskip
    Thus by combining~\eqref{eq:estimate_I1},~\eqref{eq:estimate_I2} and~\eqref{eq:estimate_I3}, we see that
    \begin{equation}
        \begin{split}
            &\norm{\Phi(z_1) - \Phi(z_2)}{C^0([0,T])} \\
            &\le \frac{C\rc F_0\sqrt{T}}{\mu}\norm{z_1 - z_2}{C^0([0,T])} \left\{\beta_1 + (\beta_1 + \beta_2 T)\sqrt{T} \left(\norm{u_0}{L^\infty(\Omega)} + F_0T\right)\right\}.
        \end{split}
    \end{equation}
    
    Now for $0\le t_1<t_2\le T$, we will estimate
    $\delta_{t_2,t_1}(\Phi(z_1)-\Phi(z_2))$.
    Since
    \begin{equation}\label{delta of the difference two Phi(z)}
    \begin{aligned}
    &\delta_{t_2,t_1}(\Phi(z_1)-\Phi(z_2)) \\
    &= \frac{1}{\mu} \int_{t_1}^{t_2}\Bigg[ \int_{\bdry{(\omegacap{z_1(\tau)})}} \gamma(u_1(\tau, y)) \nu_y \sid{y} \\
    &\phantom{=\frac{1}{\mu} \int_{t_1}^{t_2}} - \int_{\bdry{(\omegacap{z_2(\tau)})}} \gamma(u_2(\tau, y)) \nu_y \sid{y} \Bigg] \id \tau,
    \end{aligned}
    \end{equation}
    we can estimate $\delta_{t_2,t_1}(\Phi(z_1)-\Phi(z_2))$ by almost repeating the argument used to estimate $\Phi(z_1)(t)-\Phi(z_2)(t)$.
    The estimate will follow by just replacing $T$ in the previous estimates by $|t_1-t_2|$. The only Therefore
    we have
    \begin{equation}\label{1/2 Holder seminorm}
    \begin{aligned}
        &\abs{\delta_{t_1,t_2}(\Phi(z_1)-\Phi(z_2))} \\
        &\le \frac{C\beta_1 F_0{\rc}}{\mu} \abs{t_1 - t_2}^{1/2}  \left\{ \beta_1 + (\beta_1 + \beta_2)\sqrt{T} (\norm{ u_0}{L^\infty(\Omega)}+F_0T) \right\}.
    \end{aligned}
    \end{equation}
    This implies the conclusion immediately by taking sufficiently small $\rc > 0$.
    \end{proof}

\bigskip
Once having the fix point $\xc(t)=z(t)\in\ZMT$ of the mapping $\Phi$, an argument analogous to the proof of \cref{lemma:contraction} can show  that
\begin{equation}\label{xc is C1}
J(t):=\int_{\partial((\xc(t)+\omega)\cap\Omega)}\gamma(u(t,y))\nu_y\id\sigma_y\in C^0([0,T])
\end{equation}
and hence $\xc\in C^1([0,T])$. In fact for $t_1,t_2\in[0,T]$, divide $J(t_1)-J(t_2)$ into
\begin{equation}
J(t_1)-J(t_2)=\tilJ_1 +\tilJ_2+\tilJ_3,
\end{equation}
where
\begin{equation}
    \begin{aligned}
        \tilJ_1 &=\int_{(z(t_1)+\Omega)\cap\Omega}\big(\gamma'(u(t_1,y))-\gamma'(u(t_2,y))\big)\nabla u(t_1,y)\id y, \\
        \tilJ_2 &=\int_{(z(t_1)+\omega)\cap\Omega}
    \gamma'(u(t_2,y))\big(\nabla u(t_1,y)-\nabla u(t_2,y)\big)\id y, \\
        \tilJ_3 &= \int_{(z(t_1)+\omega)\cap\Omega}\gamma'(u(t_2,y))\nabla u(t_2,y)\id y - \int_{(z(t_2)+\omega)\cap\Omega}\gamma'(u(t_2,y))\nabla u(t_2,y)\id y.
    \end{aligned}
    \end{equation}
Then $\tilJ_\ell\,(\ell=1,2,3)$ are estimated as follows:
\begin{equation}
    \begin{array}{ll}
      \abs{\tilJ_1}\le\beta_2\norm{\nabla u(t_1,\cdot)}{L^\infty(\Omega)}\norm{ u(t_1,\cdot)-u_2(t_2,\cdot)}{L^\infty(\Omega)} ,\\
      \abs{\tilJ_2}\le\beta_1\norm{\nabla u(t_1,\cdot)-\nabla u(t_2,\cdot)}{L^\infty(\Omega)},\\
      \abs{\tilJ_3}\le4\beta_1\rc\norm{\nabla u(t_2,\cdot)}{L^\infty(\Omega)}\abs{z(t_1)-z(t_2)}.
    \end{array}
\end{equation}
Since $u\in C([0,T];W^{1,\infty}(\Omega))$ and
$z\in C^{1/2}([0,T])$, $J(t_1)\rightarrow J(t_2)$ as $t_1\rightarrow t_2$ and hence $J(t)\in C^0([0,T])$.

So far we have obtained that $(\xc(t),u(t,x):=u(t,x;z))$ with $z(t)=\xc(t)$ is a solution of the initial boundary value problem \eqref{eq:diffusion}--\eqref{eq:initial} satisfying all the properties given in \cref{thm:well-posedness} except the non-negativity. 

\medskip
Next we prove the uniqueness of the solution $(u,\xc)$.
Let $(u^{(j)},\xc^{(j)})$ be two solutions of~\eqref{eq:diffusion}--\eqref{eq:initial} with regularity~\eqref{regularity}. Set $v := u^{(1)}-u^{(2)}$ and $\yc := \xc^{(1)}-\xc^{(2)}$.
Then $(v,\yc)$ solves the following equations:
\begin{align}
	& \pd{t}v = D \triangle v - \alpha v + f_{\rc}(\cdot-\xc^{(1)}) - f_{\rc}(\cdot -\xc^{(2)}) & & \mbox{in}\ Q_T := (0,T) \times \Omega,  \label{eq:v-diffusion}\\
    & \pdnormal v = 0  & & \mbox{on}\ (0,T) \times \bdry{\Omega},  \\
    & v(0,\cdot) = 0 & & \mbox{in}\ \Omega, \\
	& \mu \diffop{1}{}{t}\yc = \int_{\bdry{((\xc^{(1)}+\omega)\cap\Omega)}} \gamma(u^{(1)}(\cdot,\eta))\nu_\eta \sid{\eta} & & \label{eq:yc-ode}\\
	& \phantom{\mu \diffop{1}{}{t}\yc =} - \int_{\bdry{((\xc^{(2)}+\omega)\cap\Omega)}} \gamma(u^{(2)}(\cdot,\eta))\nu_\eta  \sid{\eta} & & \mbox{in}\ (0,T), \nonumber\\
	& \yc(0) = 0.
\end{align}
By \cref{representation}, we have 
\begin{align}
    \norm{v(t,\cdot)}{L^1(\Omega)} &\le CF_0 \rc t\norm{\yc}{C^0([0,t])}, \label{eq:v0}\\
    \norm{\nabla v(t,\cdot)}{L^1(\Omega)} &\le CF_0\rc\sqrt{t}\norm{\yc}{C^0([0,t])}. \label{eq:v1}
\end{align}
Analogous to the proof of~\cref{lemma:contraction}, we write the right hand side of~\eqref{eq:yc-ode} as the sum $K_1(t) + K_2(t) + K_3(t)$ with
\begin{align*}
    K_1(t) &:= \int_{\omegacap{\xc^{(1)}(t)}} \left[\gamma'(u^{(1)}(t, y)) - \gamma'(u^{(2)}(t, y))\right]\nabla u^{(1)}(t,y) \id y, \\
    K_2(t) &:= \int_{\omegacap{\xc^{(1)}(\tau)}} \gamma'(u^{(2)}(t,y)) \nabla v(t,y) \id y, \\
    K_3(t) &:= \left( \int_{\bdry{\omegacap{\xc^{(1)}(t)}}} \gamma(u^{(2)}(t, y)) \nu_y \sid{y} 
    - \int_{\bdry{\omegacap{\xc^{(2)}(t)}}} \gamma(u^{(2)}(t, y)) \nu_y \sid{y} \right).
\end{align*}
By the mean value theorem, we have
\[
    \norm{\gamma'(u^{(1)}(t, \cdot)) - \gamma'(u^{(2)}(t, \cdot))}{L^1(\Omega)}
    \le \beta_2 \norm{v(t,\cdot)}{L^1(\Omega)},\,\,
    t\in[0,T].
\]
Hence, applying the estimate~\eqref{eq:v0}, we obtain
\begin{equation}\label{J_1 estimate}
    \abs{K_1(t)} \le C\beta_2 \rc t F_0 \left(\norm{u_0}{L^\infty(\Omega)} + F_0T\right) \norm{\yc}{C^0([0,t])},\ t\in[0,T].
\end{equation}
From the estimate~\eqref{eq:v1}, it follows that
\begin{equation}
    \abs{K_2(t)} \le C\beta_1 F_0 \rc \sqrt{t} \norm{\yc}{C^0([0,t])},\,\,t\in[0,T].
\end{equation}

Similarly to~\eqref{eq:estimate_I3}, $K_3(t)$ is estimated as
\begin{equation}
    \abs{K_3(t)} \le C \beta_1 \rc t \left(\norm{u_0}{L^\infty(\Omega)} + F_0T\right) \norm{\yc}{C^0([0,t])},\,\,t\in[0,T].
\end{equation}
Thus we see that for any $t\in[0,T]$ we have
\begin{equation}\label{y_c estimate}
    \begin{aligned}
        \yc(t)&=\frac{1}{\mu} \int_0^t (K_1(\tau) + K_2(\tau) + K_3(\tau)) \id \tau \\
        &\le \frac{C\rc}{\mu} \max\{\beta_1,\beta_2\} \\
        &\phantom{=} \times \int_0^t \left( \tau (1+F_0) \left(\norm{u_0}{L^\infty(\Omega)} + F_0T\right) + F_0\sqrt{\tau} \right)\norm{\yc}{C^0([0,\tau])}\id \tau.
    \end{aligned}
\end{equation}
Hence from Gronwall's inequality, it follows that $\yc=0$ in $[0,T]$, and by \eqref{eq:v0} we have $v = 0$. This completes the proof of the uniqueness. 

\medskip
Finally we will prove the non-negativity of $u$. To begin with, we replace the initial data $u_0(x)$ by $u_0^\eta:=u_0(x)+\eta$ with an arbitrarily small fix constant $\eta>0$. Note that $u_0^\eta\in W^{1,\infty}(\Omega)$ satisfies the compatibility condition and satisfies
\begin{equation}\label{strict positivity}
u_0^\eta\ge\eta\,\,\text{on}\,\,\overline\Omega.
\end{equation}
Denote the corresponding solution of the 
initial boundary value problem \eqref{eq:diffusion}--\eqref{eq:initial} by $(u^\eta(t,x), \xc^\eta(t))$. Let $v^\eta=u^\eta-u$ and $y^\eta=\xc^\eta-\xc$. Then we can almost repeat the argument used to prove the uniqueness to show that
\begin{equation}\label{estimate of v^eta}
\norm{v^\eta}{L^\infty([0,T]; W^{1,1}(\Omega))}=O(\eta),\,\,0<\eta\ll 1.
\end{equation}
In fact the estimate corresponding to \eqref{eq:v0} changes to
\[
\norm{v^\eta(t,\cdot)}{L^1(\Omega)}\le C\left(\eta+F_0 \rc t\norm{\yc}{C^0([0,t])}\right),\ t\in[0,T].
\]
Consequently we need to change the estimate corresponding to \eqref{J_1 estimate} to
\[
\abs{J_1^\eta(t)}\le C\beta_2\left(\norm{ u_0}{L^\infty(\Omega)}+F_0T\right)\left(\eta+F_0{\rc}t\norm{ \yc^\eta}{C([0,t])}\right),\ t\in[0,T].
\]
Hence the estimate corresponding to \eqref{y_c estimate} changes to
\[
    \begin{aligned}
        \abs{\yc^\eta(t)} &\le \frac{C\rc}{\mu} \max\{\beta_1,\beta_2\} \\
        &\phantom{\le} \times \int_0^t \{ \tau (1+F_0) (\norm{u_0}{L^\infty(\Omega)}
    + F_0T)+F_0\sqrt{\tau} \}\norm{\yc}{C^0([0,\tau])}\id \tau \\
        &\phantom{\le} +C\beta_2\eta(\norm{ u_0}{L^\infty(\Omega)}+F_0T)t
    \end{aligned}
\]
for any 
$t\in[0,T]$.
Then Gronwall's inequality immediately implies \eqref{estimate of v^eta}.
Now we invoke $\norm{ v^\eta}{C^0([0,T];W^{1,\infty}(\Omega))}$ for $0<\eta\ll 1$ is bounded to show 
$\norm{ v^\eta}{C^0([0,T];L^\infty(\Omega))}\rightarrow 0$ as $\eta\rightarrow 0$. By the Garliardo--Nirenberg inequality, there exist positive constants $c_1,\,c_2$ such that for any $t\in[0,T]$
$$
\norm{ v^\eta(t,\cdot)}{L^\infty(\Omega)}\le c_1\norm{\nabla v^\eta}{L^\infty(\Omega)}^{2/3}\norm{ v^\eta}{L^1(\Omega)}^{1/3}+c_2\norm{ v^\eta}{L^1(\Omega)}.
$$ 
This immediately implies $\norm{ v^\eta}{C^0([0,T];L^\infty(\Omega))}\rightarrow 0$ as $\eta\rightarrow 0$.
Therefore it is enough to prove the non-negativity of $u^\eta$. 

Let $0\le\rho_\epsilon(x)\in C^\infty_0(\{\abs{x}\le\epsilon\})$ with $0<\epsilon\ll 1$ be a mollifier.
Fix any small $\epsilon_0>0$.
Consider $u_\epsilon^\eta=u^\eta\ast\rho_\epsilon$ for $0<\epsilon<\epsilon_0$.
Then each $u_\epsilon^\eta$ satisfies
\begin{equation}
    (\partial_t-D\Delta+\alpha)u_\epsilon^\eta=f_{\rc}(\cdot-\xc)\big|_{Q_T}\ast\rho_\epsilon:=g_\epsilon\,\,\text{in}\,\,Q_T^{\epsilon_0},
\end{equation}
where $Q_T^{\epsilon_0}=(0,T)\times\Omega_{\epsilon_0}$ with $\Omega_\epsilon:=\{x\in\Omega: \abs{x-y}>\epsilon,\, y\in\partial\Omega\}$.
By $u^\eta\in C^1((0,T];L^p(\Omega))$, the mollified function $u_\epsilon^\eta$ is $C^1$ with respect to $t$ and $C^2$ with respect to $x$ in
$Q_T^{\epsilon_0}$ and $g_\epsilon\ge0$ in $Q_T^{\epsilon_0}$.
Further $u_\epsilon^\eta(\delta,\cdot)>0$ in $\Omega_\epsilon$ for any $0<\delta\ll 1$.
Hence by the maximum principle (see \cite[Chapter 2, Section 1, Theorem 1]{Friedman} ), we have $u_\epsilon^\eta\ge0$ in $(\delta,T)\times\Omega_\epsilon$.
Then by $u^\eta\in C^0([0,T];W^{1,\infty}(\Omega))$, $u_\epsilon^\eta(t,x)\rightarrow u^\eta(t,x)$ as $\epsilon\rightarrow 0$ at every $(t,x)\in (\delta,T)\times\Omega_{\epsilon_0}$.
Hence $u^\eta\ge0$ in $(\delta,0)\times\Omega_{\epsilon_0}$.
Since $\delta$ and $\epsilon_0$ are arbitrarily small positive numbers and $u^\eta$ is continuous in $[0,T]\times\overline\Omega$,
we have $u^\eta\ge0$ on $[0,T]\times\overline\Omega$.
This completes the proof of \cref{thm:well-posedness}.

\medskip
\section{Derivation of the short time approximate solution}\label{sec:asymptotics}${}$\\
In this section, we will derive the short time approximate solution $(\tilu,\tilxc)$ stated in \cref{thm:main}. Let us start this formal derivation by introducing the scaled source function $F(x)$ via
\begin{equation}
    f_{\rc}(x) = F\left(\frac{x}{\rc}\right).
\end{equation}
Also, following the idea of~\cite{Constantinescu2010,Nakagawa2014}, we introduce the ``microscopic scaling'' around the fixed time $t_0$ and $x_0 := \xc(t_0)$ by
\begin{equation}
    x = x_0 + \rc(\xi - x_0),\  t = t_0 + \rc^2 (\tau - t_0)
\end{equation}
which maps a neighborhood of $(x_0,t_0)$ in the $(x,t)$-space to a very large neighborhood of $(x_0,t_0)$ in the $(\xi,\tau)$-space likewise a microscope due to the smallness of $\rc$.
Further by setting $\overline{u}_0:=u(x_0,t_0)$, denote the re-scaled concentration $U(\xi,\tau)$ and the position $\xi_c(\tau)$ of the camphor by
\begin{equation}
    \begin{split}
        U(\xi,\tau) &:= \frac{u(x_0 + \rc(\xi - x_0), t_0 + \rc^2 (\tau - t_0)) - \overline{u}_0}{\rc^2}, \\
        \xic(\tau) &:= x_0 + \frac{\xc(t_0 + \rc^2 (\tau - t_0))-x_0}{\rc},
    \end{split}
\end{equation}
which immediately implies $U(x_0,t_0)=0$ and $\xic(t_0) = x_0$.
Then $U$ and $\xic$ satisfy the following equation:
\begin{align}
    & \pd{\tau}U = D \triangle_\xi U - \alpha (u_0 + \rc^2 U) + F(\xi-\xic(\tau)), \label{eq:U}\\
    & \mu \diffop{1}{}{\tau}\xic = \rc^2 \int_{\bdry{B_1(0)}} \gamma(u_0 + \rc^2U(\cdot,\xic+\eta))\nu_\eta \sid{\eta}, \label{eq:xic}
\end{align}
where $B_1(0)$ is the unit open disk centered at $0$.

Now we assume that $U$ has expansions in $\rc$:
\begin{equation}
        U = U_0 + \rc^2 U_2 + \dots.
\end{equation}
Substituting the expansion of $U$ into the diffusion equation \eqref{eq:U} and setting the terms with same powers of $\rc$ equal zero,
we have
\begin{equation}\label{eq:U0_and_U2}
    \begin{split}
        \pd{\tau} U_0 &= D\triangle_\xi U_0 - \alpha u_0 + F(\xi - \xic(\tau)), \\
        \pd{\tau} U_2 &= D\triangle_\xi U_2 - \alpha U_0.
    \end{split}
\end{equation}
Then using the heat kernel
\begin{equation}
    \Gamma_\tau = \frac{1}{4\pi D \tau} \exp \left[ - \frac{\abs{\xi}^2}{4 D \tau}\right],\,\,\tau>0,
\end{equation}
$U_0$ and $U_2$ have the following representations:
\begin{equation}
    \begin{split}
        U_0(\tau,\cdot) &= \Gamma_{\tau - t_0} \ast \Utzero + \int_{t_0}^{\tau} \Gamma_{\tau - \tau'} \ast F(\cdot - \xic(\tau'))\id \tau' - \alpha u_0 (\tau - t_0), \\
        U_2(\tau,\cdot) &= -\alpha \int_{t_0}^{\tau} \Gamma_{\tau - \sigma} \ast U_0(\tau',\cdot) \id \tau',
    \end{split}
\end{equation}
where $\Utzero := U(t_0,\cdot) = u(t_0,\cdot)$ and ``$\ast$'' denotes the convolution with respect to the $\xi$ variable.
Further, by applying the semigroup property of the heat kernel, we have
\begin{equation}
\begin{aligned}
    U_2(\tau,\cdot)
    &= -\alpha\int_{t_0}^{\tau} \Gamma_{\tau - \tau'} \ast \left(\Gamma_{\tau' - t_0} \ast \Utzero\right) \id \tau' \\
    &\phantom{=} - \alpha\int_{t_0}^{\tau} \Gamma_{\tau - \tau'} \ast \left( \int_{t_0}^{\tau'} \Gamma_{\tau' - \tau''} \ast F(\cdot - \xic(\tau''))\id \tau'' \right) \id \tau' \\
    &\phantom{=} + \frac{1}{2}\alpha^2 u_0 (\tau - t_0)^2 \\
    &= -\alpha\int_{t_0}^{\tau} \Gamma_{\tau - t_0} \ast \Utzero \id \tau' %
    - \alpha\int_{t_0}^{\tau} \left( \int_{t_0}^{\tau'} \Gamma_{\tau - \tau''} \ast F(\cdot - \xic(\tau''))\id \tau'' \right) \id \tau' \\
    &\phantom{=} + \frac{1}{2}\alpha^2 u_0 (\tau - t_0)^2 \\
    &= -\alpha(\tau - t_0) \Gamma_{\tau - t_0} \ast \Utzero %
    - \alpha\int_{t_0}^{\tau} (\tau - \tau') \Gamma_{\tau-\tau'} \ast F(\cdot - \xic(\tau'))\id \tau' \\
    &\phantom{=} + \frac{1}{2}\alpha^2 u_0 (\tau - t_0)^2.
\end{aligned}
\end{equation}

Now recall the integral formula
\begin{displaymath}
    \frac{1}{\sqrt{\pi\lambda}} \int_{-\infty}^{\infty} \exp\left(-\frac{y^2}{\lambda} \right) y^\ell \id y =
    \left\{
        \begin{aligned}
                & 0 & &\mbox{if $\ell$ is odd,} \\
                & 1 & &\mbox{if $\ell = 0$,} \\
                & \left(\frac{\lambda}{2}\right)^m (2m-1)!! & &\mbox{if $\ell = 2m\,(m = 1,2,\dots)$}.
        \end{aligned}
    \right.
\end{displaymath}
Then the convolution $\Gamma_{\tau - t_0} \ast \Utzero$ can be written in the form 
\begin{displaymath}
    \Gamma_{\tau - t_0} \ast \Utzero = \Utzero + D(\tau - t_0) \triangle_\xi \Utzero + \frac{1}{2} D^2 (\tau - t_0)^2 \triangle_\xi^2 \Utzero + o((\tau - t_0)^2)\, (\tau \to t_0).
\end{displaymath}
Similar argument yields that
\begin{equation*}
    \begin{aligned}
    \Gamma_{\tau - \tau'} \ast F(\cdot - \xic(\tau')) &= F(\cdot - \xic(\tau')) + D(\tau - \tau')\triangle_\xi F(\cdot - \xic(\tau')) \\
    &\phantom{=} + \frac{1}{2} D^2 (\tau - \tau')^2 \triangle_\xi^2 F(\cdot - \xic(\tau')) + o((\tau - t_0)^2)\, (\tau \to t_0).
    \end{aligned}
\end{equation*}
Hence $U_0$ and $U_2$ admit the approximations given as
\begin{equation}
    \begin{split}
        U_0 &= \Utzero %
        + D(\tau - t_0) \triangle_\xi\Utzero %
        + \frac{1}{2}D^2(\tau - t_0)^2 \triangle_\xi^2  \Utzero \\
        &\phantom{==}+ \int_{t_0}^{\tau} \left(%
            F(\cdot - \xic(\tau')) %
            + D(\tau - \tau') \triangle_\xi F(\cdot - \xic(\tau')) %
            \right)\id \tau' \\
        &\phantom{==} - \alpha u_0 (\tau - t_0) + o((\tau - t_0)^2), \\
        U_2 &= %
        -\alpha (\tau - t_0) \Utzero %
        -\alpha (\tau - t_0)^2 \triangle_\xi\Utzero \\
        &\phantom{==}- \alpha \int_{t_0}^{\tau} (\tau - \tau') F(\cdot - \xic(\tau'))\id \tau' + \frac{1}{2}\alpha^2 u_0 (\tau - t_0)^2 + o((\tau - t_0)^2).%
    \end{split}
\end{equation}
Then the approximate solution $\tilU(\tau,\xi) = u_0 + \rc^2(U_0 + \rc^2 U_2)$ is given as
\begin{equation}\label{tilde U}
    \begin{split}
        \tilU(\tau,\xi) &= \rc^2\Utzero + \rc^2(\tau - t_0)(D\triangle - \alpha\rc^2)\Utzero + \rc^2(\tau - t_0)^2 \left(\frac{1}{2}D^2 \triangle^2 - \alpha \rc^2 \triangle \right)\Utzero \\
        &\phantom{=} + \rc^2 \int_{t_0}^{\tau} \left[ F(\xi - \xic(\tau')) + (\tau - \tau')(D\triangle - \rc^2 \alpha)F(\xi - \xic(\tau')) \right] \id \tau' \\
        &\phantom{=} + u_0 - \alpha \rc^2 u_0(\tau - t_0) + \frac{1}{2} \alpha^2 \rc^4u_0 (\tau - t_0)^2 + o((\tau - t_0)^2)
    \end{split}
\end{equation}
Translating \eqref{tilde U} in terms of the original variable $(x,t)$, we obtain
\begin{equation}
    \begin{split}
        \tilu(t,x) &= \rc^2u^{t_0} + \rc^2(t - t_0)(D\triangle - \alpha)u^{t_0} + \rc^2(t - t_0)^2 \left(\frac{1}{2}D^2 \triangle^2 - \alpha \triangle \right)u^{t_0} \\
        &\phantom{=} + \int_{t_0}^{t} \left[ f_{\rc}\left(\frac{x - \xic(t')}{\rc}\right) + (t - t')(D\triangle - \alpha)f_{\rc}\left(\frac{x - \xic(t')}{\rc}\right) \right] \id \tau' \\
        &\phantom{=} + u_0 - \alpha u_0(t - t_0) + \frac{1}{2} \alpha^2 u_0 (t - t_0)^2 + o(\rc^{-2}(t - t_0)^2)\,.
    \end{split}
\end{equation}
Then substituting $U = \tilU$ into equation~\eqref{eq:xic} and using the approximation of $\gamma$ given by 
\begin{displaymath}
    \gamma(\tilU) = \gamma(u_0) + \gamma'(u_0) \rc^2(U_0 + \rc^2 U_2), 
\end{displaymath}
we obtain the following ODE for $\xc(t)$:
\begin{displaymath}
    \begin{array}{ll}
    \displaystyle \mu \diffop{1}{}{\tau} \xic \\
    \displaystyle \quad = - \rc^4 \gamma_1 \{\Vtzero(\xic(\tau)) + (\tau - t_0)(D\triangle - \alpha\rc^2)\Vtzero(\xic(\tau))\\
    \displaystyle \phantom{=}\quad + (\tau - t_0)^2 \left(\frac{1}{2}D^2 \triangle^2 - \alpha \rc^2 \triangle \right)\Vtzero(\xic(\tau)) \} \\
    \displaystyle \phantom{=}\quad - \rc^4 \gamma_1 \left( \int_{t_0}^{\tau} \left[ \Phi(\xic(\tau) - \xic(\tau')) + (\tau - \tau')(D\triangle - \rc^2 \alpha)\Phi(\xic(\tau) - \xic(\tau')) \right] \id \tau' \right) \\
    \displaystyle \phantom{=}\quad + o(\rc^2(\tau - t_0)^2),
    \end{array}
\end{displaymath}
where we have introduced
\begin{equation*}
    \begin{split}
        \Vtzero(\xi) &:=  \int_{\bdry{B_1(0)}} \Utzero(\xi + \eta) \nu_\eta \sid{\eta},\\
        \Phi(\xi) &:= \int_{\bdry{B_1(0)}} F(\xi + \eta) \nu_\eta \sid{\eta}. 
    \end{split}
\end{equation*}
In terms of the original variable $(t,x)$ this becomes
\begin{displaymath}
    \begin{aligned}
        \mu \diffop{1}{}{t} \xc(t) &= %
        - \gamma_1 \rc \left( v^{t_0}(\xc(t)) + (t - t_0)A_1v^{t_0}(\xc(t)) + (t - t_0)^2 A_2 v^{t_0}(\xc(t))\right) \\
        &\phantom{=}%
        - \gamma_1 \rc \left( \int_{t_0}^{t} \left[ \varphi\left(\xc(t) - \xc(t')\right) + (t - t')A_1\varphi\left(\xc(t) - \xc(t')\right) \right] \id t' \right) \\
    &\phantom{=}+ o(\rc^{-2}(t - t_0)^2)\,,
    \end{aligned}
\end{displaymath}
where
\begin{displaymath}
    \begin{aligned}
        v^{t_0}(x) &:= \int_{\bdry{B_{\rc}(0)}} U(t_0,x + y) \nu_y \id \sigma_y, \\
        \varphi(x) &:= \int_{\bdry{B_{\rc}(0)}} f_{\rc} \left(x + y\right)\nu_y \id \sigma_y, \\
        A_1 &:= D\triangle_x - \alpha, \\
        A_2 &:= \frac{1}{2}D^2\triangle_x^2 - \alpha \triangle_x.
    \end{aligned}
\end{displaymath}
\section{Error estimate for the short time approximate solution}\label{sec:error}${}$\\

In this section, we give the error estimate for the short time approximate solution $(\tilu,\tilxc)$ to complete the proof of \cref{thm:main}.
First of all from~\eqref{eq:U0_and_U2}, observe that $\tilU$ satisfies the following equations:
\begin{equation}
	\begin{split}
    	&\pd{\tau} \tilU = D \triangle_\xi \tilU - \alpha(u_0 + \rc^2 \tilU) + \alpha \rc^4 U_2 - F(\cdot - \xic(\tau)), \\
        &\tilU(\cdot,t_0) = \Utzero.
    \end{split}
\end{equation}
Hence $R := \rc^{-2} (U - \tilU)$ and $\zeta := \rc^{-2}(\xic - \tilxic)$ satisfy
the Cauchy problem
\begin{equation}
	\begin{split}
		\pd{\tau} R &= D \triangle_\xi R - \alpha R - \alpha \rc^2 U_2 - \rc^{-2}\left[F(\cdot - \xic(\tau)) - F(\cdot - (\xic(\tau)-\rc^2 \zeta(\tau)))\right], \\
        \mu \diffop{1}{}{\tau} \zeta &= - \rc^2 \beta_1 \int_{\bdry{B_1(0)}} R(\xic+\eta,\cdot)\nu_\eta \id s_\eta \\
        &\phantom{=} -\beta_1 \int_{\bdry{B_1(0)}} \left[ \tilU(\xic+\eta,\cdot) - \tilU(\xic - \rc^2\zeta + \eta,\cdot) \right]\nu_\eta \id s_\eta,
	\end{split}
\end{equation}
with $R(\cdot,t_0)=\zeta(t_0) = 0$.
In order to estimate the error of our approximate solution, 
it is enough to obtain the uniform boundedness of $\zeta$ and $R$ with respect to $\rc$.

Based on this, transform the above Cauchy problem for $(R,\zeta)$ to the following system of integral equations:
\begin{equation}
\begin{split}
	R(\tau,\cdot) &= \int_{t_0}^{\tau} \int_\Omega \Gamma(\tau,\cdot;s,\eta) \\
	&\phantom{\int_{t_0}^{\tau} \int_\Omega} \times \left[ - \alpha \rc^2 U_2(s,y) + \frac{F(\eta - \xic(s)) - F(\eta - (\xic(s)-\rc^2 \zeta(s)))}{\rc^2} \right] \id \eta \id s,\\
	\zeta(\tau) &= -\frac{\beta_1}{\mu} \int_{t_0}^{\tau} \left(\rc^2\int_{\bdry{B_1(0)}} R(s,\xic+\eta)\nu_\eta \id s_\eta\right. \\
    &\phantom{-\frac{\beta_1}{\mu} \int_{t_0}^{\tau}} \quad \left. + \int_{\bdry{B_1(0)}} \left[ \tilU(s,\xic+\eta) - \tilU(s,\xic - \rc^2\zeta + \eta) \right]\nu_\eta \id s_\eta \right)\id s.
\end{split}
\end{equation}
Then the estimate of $(R,\zeta)$ follows from the unique solvability of this system of integral equations with an estimate which can be shown by the successive approximation argument. More precisely we show the existence of a solution and its estimate by proving the convergence of the sequence $\{(R_n,\zeta_n)\}_{n=0,1,\dots}$ defined by  
\begin{equation}\label{eq:dfn_R_zeta}
	\begin{split}
		R_{n+1}(\tau,\cdot) &= \int_{t_0}^{\tau} \int_\Omega \Gamma(\tau,\cdot;s,\eta) \\
	    &\phantom{\int_{t_0}^{\tau} \int_\Omega} \times \left[ - \alpha \rc^2 U_2(s,y) + \frac{F(\eta - \xic(s)) - F(\eta - (\xic(s)-\rc^2 \zeta_n(s)))}{\rc^2} \right] \id \eta \id s,\\
	    \zeta_{n+1}(\tau) &= -\frac{\beta_1}{\mu} \int_{t_0}^{\tau} \left(\rc^2\int_{\bdry{B_1(0)}} R_{n+1}(s,\xic+\eta)\nu_\eta \id s_\eta\right. \\
        &\phantom{-\frac{\beta_1}{\mu} \int_{t_0}^{\tau}} \quad \left. + \int_{\bdry{B_1(0)}} \left[ \tilU(s,\xic+\eta) - \tilU(s,\xic - \rc^2\zeta_n + \eta) \right]\nu_\eta \id s_\eta \right)\id s
	\end{split}
\end{equation}
with $(R_0,\zeta_0)=(0,0)$. Then estimates similar to \eqref{estimate of R} and \eqref{estimate of zeta} showing the convergence of $\{(R_n,\zeta_n)\}_{n=0,1,\cdots}$ can show that the solution of the system of integral equations is unique, which implies the estimate of $(R,S)$. 

By the definition of $R_{n+1}$, we have
\begin{equation}
    \begin{aligned}
        &\norm{R_{n+1}(\tau,\cdot)}{W^{1,1}(\Omega)} \\
        &\, \le C \sqrt{\tau - t_0} \\
        &\phantom{\le} \, \times \left(\alpha \rc^2 \norm{U_2}{C^0([0,T];L^1(\Omega))}  + \norm{\frac{F(\cdot - \xic) - F(\cdot - (\xic-\rc^2 \zeta_n))}{\rc^2}}{C^0([0,T];L^1(\Omega))}\right) \\
        &\,\le C \sqrt{\tau - t_0} \left(\alpha \rc^2 \norm{U_2}{C^0([0,T];L^1(\Omega))} + \norm{F}{W^{1,1}(\Omega)}\norm{\zeta_n}{C^0([0,T])}\right).
    \end{aligned}
\end{equation}
Here we have used~\cref{lemma:FS} and Young's inequality.
So $\zeta_{n+1}$ can be estimated as
\begin{align*}
    &\abs{\zeta_{n+1}(\tau)} \\
    &\le
	\frac{\beta_1}{\mu} \int_{t_0}^{\tau} \Bigg[\rc^2\int_{\bdry{B_1(0)}} \abs{R_{n+1}(s,\xic+\eta)} \id s_\eta \\
    &\phantom{\le
	\frac{\beta_1}{\mu} \int_{t_0}^{\tau}} %
	+ \int_{\bdry{B_1(0)}}\quad \abs{ \tilU(s,\xic+\eta) - \tilU(s,\xic - \rc^2\zeta_n + \eta) }\nu_\eta \id s_\eta \Bigg]\id s \\
    &\le \frac{\beta_1\pi\rc^2}{\mu} \int_{t_0}^{\tau} \left[\norm{R_{n+1}(s,\cdot)}{W^{1,1}(\Omega)}  + \norm{\tilU(s,\cdot)}{C^1(\Omega)} \abs{\zeta_{n}(s)} \right] \id s \\
    &\le \frac{C\beta_1\pi\rc^2}{\mu} \Bigg[ \int_{t_0}^{\tau} \sqrt{s - t_0} \left(\alpha \rc^2 \norm{U_2}{C^0([0,T];L^1(\Omega))} + \norm{F}{W^{1,1}(\Omega)}\norm{\zeta_n}{C^0([0,T])}\right) \id s \\ 
    &\phantom{\le \frac{C\beta_1\pi\rc^2}{\mu} }\quad %
    + \int_{t_0}^{\tau} \norm{\tilU(s,\cdot)}{C^1(\Omega)} \abs{\zeta_{n}(s)} \id s \Bigg]\\
    &\le \frac{C\beta_1\pi\rc^2}{\mu} M (\tau - t_0) \left( \sqrt{\tau - t_0} + \norm{\zeta_n}{C^0([0,T])} \right),
\end{align*}
where $M := \max \left\{\alpha\norm{U_2}{C^0([t_0,\tau];L^p(\Omega))}, \norm{F}{W^{1,1}}, \norm{\tilU}{C^0([t_0,\tau];W^{2,p}(\Omega)}\right\}$.

Then the sequence $\{(R_n,\zeta_n)\}$ satisfies the following estimates:
\begin{equation}
    \begin{split}
        \norm{\zeta_1}{C^0([t_0,\tau])} &\le \frac{C\beta_1\pi\rc^2}{\mu} M (\tau - t_0)^{3/2}, \\
        \norm{\zeta_{n+1}}{C^0([t_0,\tau])} &\le \frac{C\beta_1\pi\rc^2}{\mu} M (\tau - t_0) \left( \sqrt{\tau - t_0} + \norm{\zeta_n}{C^0([0,T])} \right), \\
        \norm{R_{1}(\tau,\cdot)}{W^{1,1}} &\le CM\rc^2 \sqrt{\tau - t_0} \\
        \norm{R_{n+1}(\tau,\cdot)}{W^{1,1}} &\le CM\sqrt{\tau-t_0}\left(\rc^2+\norm{\zeta_n}{C^0([t_0,\tau])}\right)
    \end{split}
\end{equation}
Hence for arbitrarily fixed $B > 0$, taking sufficiently small $\delta$ satisfying
\begin{equation}\label{eq:rel_dalta_B}
        \delta \le \min \left\{
            \frac{\mu}{CM\beta_1\pi\rc^2},\,
            B^2,\,
            \frac{1}{2A},\,
            \frac{B}{CM(\rc^2+B)}
            \right\}
\end{equation}
we obtain $\norm{\zeta_{n}}{C^0([t_0,\tau])}\le B$ and $\norm{R_{n}(\tau,\cdot)}{W^{1,1}} \le B$ for all $n$ if $0 \le \tau - t_0 \le \delta$.

What is left is to show is the convergence of $\{\zeta_n\}$ in $C^0([t_0,\tau])$.
By~\eqref{eq:dfn_R_zeta}, we have
\[
\begin{aligned}
&R_{n+1}(\tau,\cdot) - R_{n}(\tau,\cdot)\\
&= \int_{t_0}^{\tau}\!\!\int_\Omega \Gamma(t,\cdot;s,y) \left[\frac{F(y - (\xic(s)-\rc^2 \zeta_{n}(s))) - F(y - (\xic(s)-\rc^2 \zeta_{n-1}(s)))}{\rc^2} \right] \id y\id s.
\end{aligned}
\]
Then by the estimate~\eqref{eq:Gamma-0},~\eqref{eq:Gamma-1} and Young's inequality, we have
\begin{equation}\label{estimate of R}
    \begin{split}
        &\norm{R_{n+1}(\tau,\cdot) - R_{n}(\tau,\cdot)}{W^{1,1}(\Omega)} \\
        &\le C\sqrt{\tau - t_0} \norm{\frac{F(\cdot - (\xic-\rc^2 \zeta_{n})) - F(\cdot - (\xic-\rc^2 \zeta_{n-1}))}{\rc^2} }{L^1(t_0,\tau;W^{1,1}(\Omega))} \\
        &\le CM\sqrt{\tau - t_0}\int_{t_0}^{\tau} \abs{\zeta_{n}(s) - \zeta_{n-1}(s)} \id s.
    \end{split}
\end{equation}
Hence by~\eqref{eq:dfn_R_zeta}, we have
\begin{equation}\label{estimate of zeta}
    \begin{split}
        &\abs{\zeta_{n+1}(s)-\zeta_n(s)} \\
        &\le \frac{\beta_1}{\mu} \int_{t_0}^{s} 
        \left( \rc^2 \int_{B_1(0)} \abs{\nabla R_{n+1}(s',\xic+\eta)-\nabla R_{n}(s',\xic+\eta)} \id \eta \right. \\
        &\phantom{\le} \left. + \int_{B_1(0)} \abs{\nabla\tilU(s',\xic- \rc^2\zeta_{n}+\eta) - \nabla\tilU(s',\xic - \rc^2\zeta_{n-1} + \eta)} \id \eta \right)\id s' \\
        &\le \frac{CM\beta_1\pi \rc^2}{\mu}  \int_{t_0}^{s} 
        \sqrt{s' - t_0}\left(\int_{t_0}^{s'}  \abs{\zeta_n(s'')-\zeta_{n-1}(s'')} \id s''\right)\id s' \\
        &\phantom{\le\frac{CM\beta_1\pi \rc^2}{\mu}} + \int_{t_0}^s \abs{\zeta_{n}(s') - \zeta_{n-1}(s')} \id s' \\
        &\le \frac{CM\beta_1\pi \rc^2}{\mu}
        \int_{t_0}^{s}
        \left( 1 + (\tau - t_0)^{1/2} (s - s') \right)
        \abs{\zeta_n(s')-\zeta_{n-1}(s')}\id s'.
    \end{split}
\end{equation}
Finally we obtain that
\begin{displaymath}
    \abs{\zeta_{n+1}(s)-\zeta_n(s)} \le \frac{CM\beta_1\pi \rc^2}{\mu} \cdot \frac{(\tau - t_0)^{n}}{n!}\abs{\zeta_1(s) - \zeta_0(s)}
\end{displaymath}
and the convergence follows from this inequality.
\section{Numerical Simulation}\label{sec:numerical}${}$\\
\begin{figure}[tbp]
    \centering
    \includegraphics[width=\linewidth]{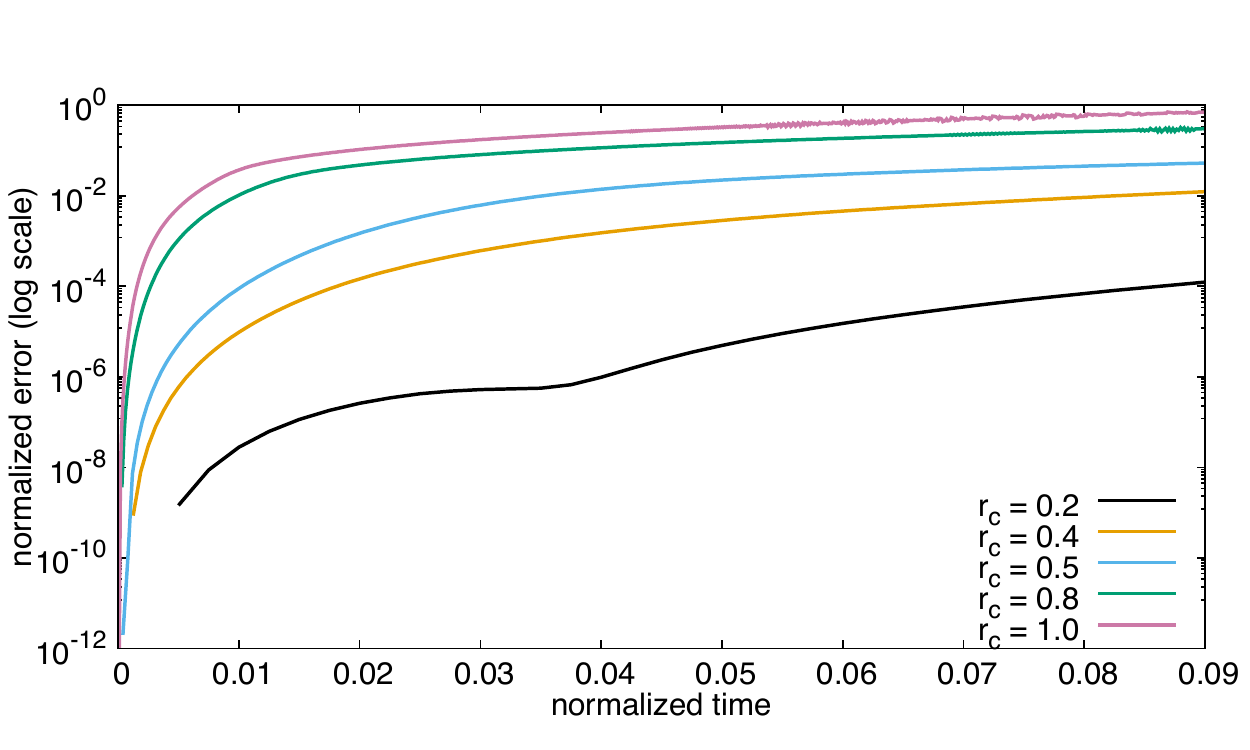}
    \includegraphics[width=\linewidth]{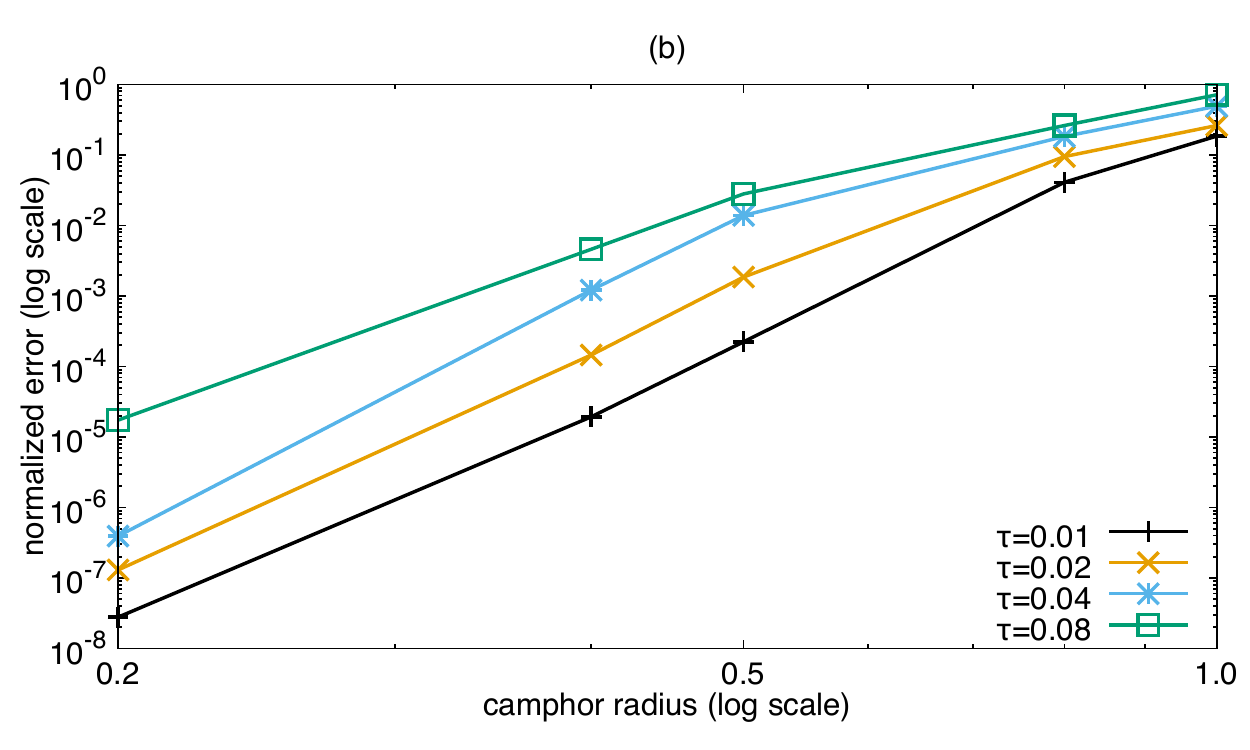}
    \caption{(a) The normalized error $\abs{\xc-\tilxc}/\rc$ of camphor postion between the solution of the original system and the approximate solution with repsect to the normalized time $\tau=t/\rc^2$.
    (b) The normalized error with respect to the camphor radius for fixed normalized times $\tau = 0.01, 0.02, 0.04, 0.08$.
    }
    \label{fig:rc-error}
\end{figure}
\begin{figure}[tbp]
    \includegraphics[width=0.45\linewidth]{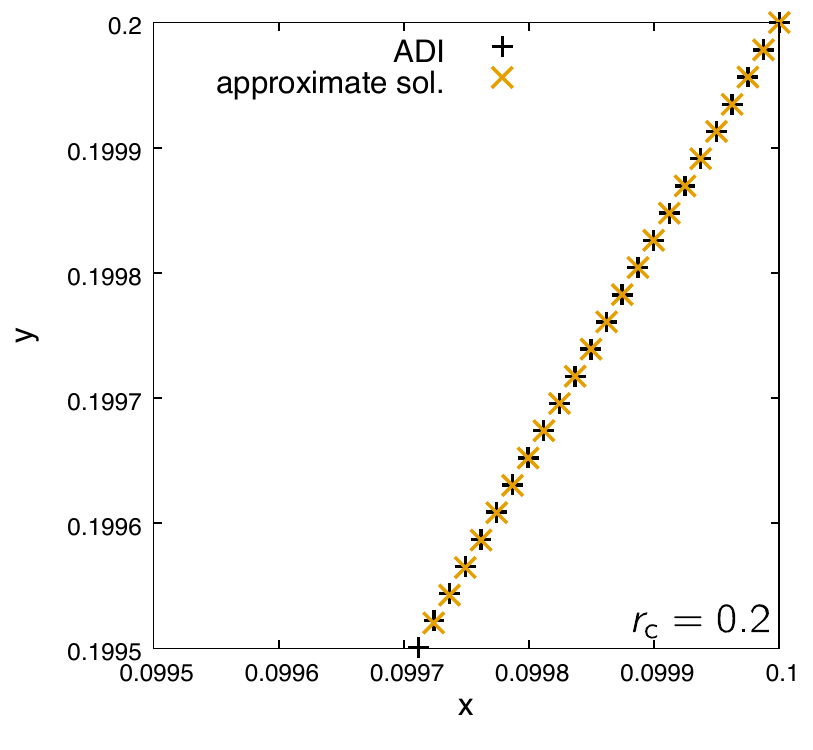}
    \includegraphics[width=0.45\linewidth]{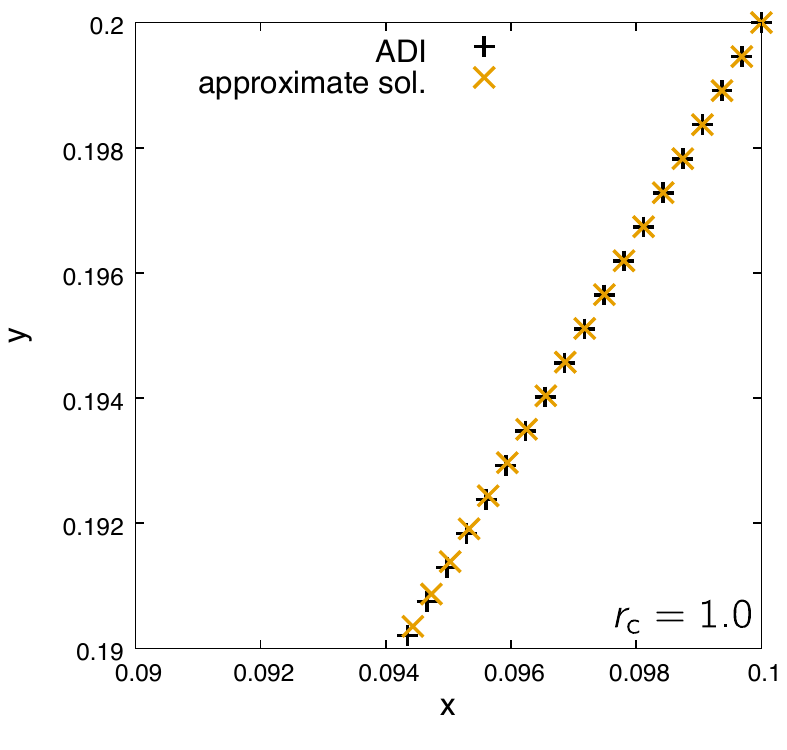}
    \caption{The orbit of the camphor particle calculated by ADI scheme and our approximate equation.
    The initial position of the camphor is $(0.1,0.2)$, the upper right corner of each graphs.
    Each point shows the position of the camphor in every step
    and we set the timestep $\Delta t = 1.0 \times 10^{-4}$.}
    \label{fig:orbits}
\end{figure}
\begin{figure}[tbp]
\includegraphics[width=\linewidth]{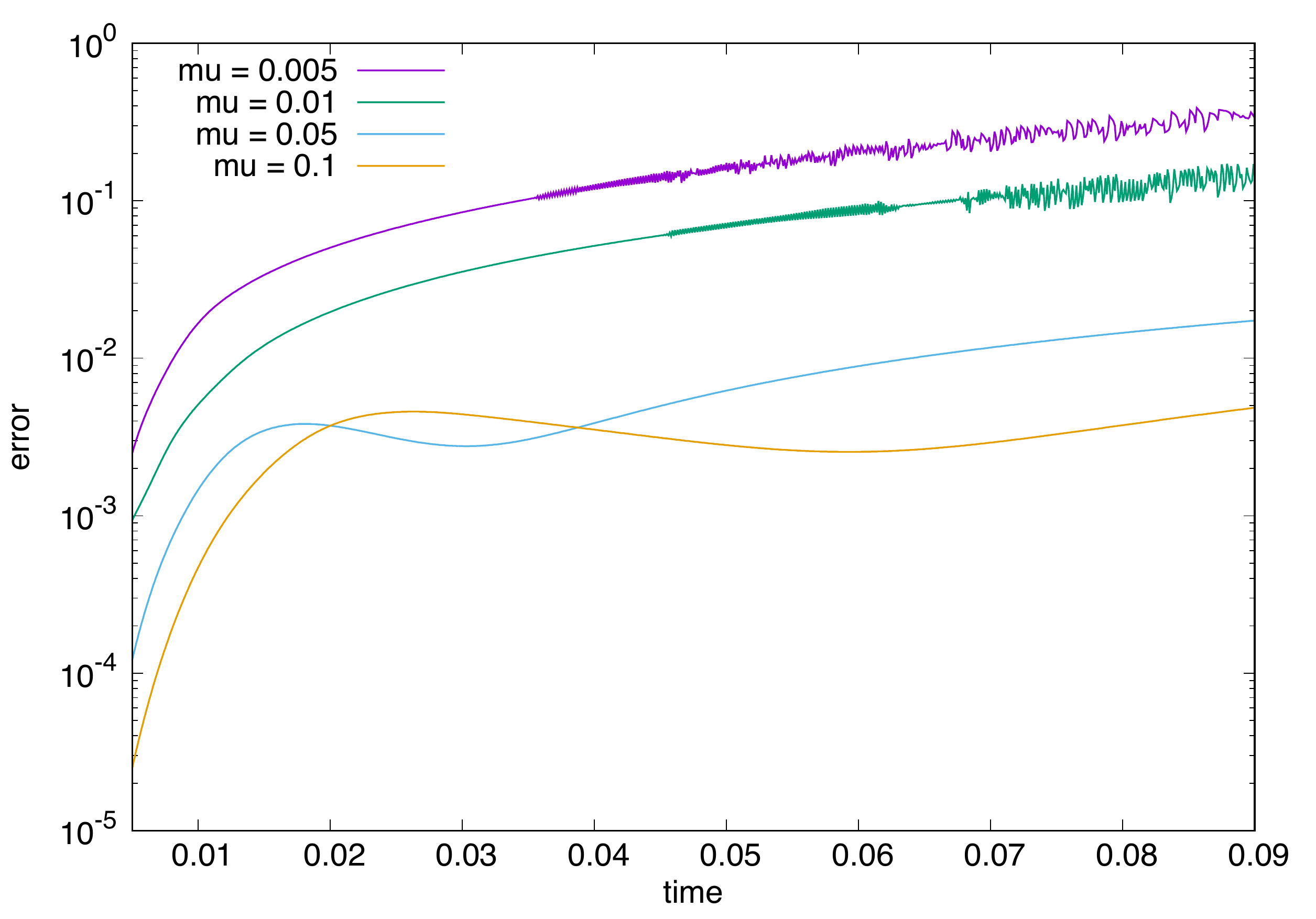}
\caption{The error $\abs{\xc-\tilxc}$ with respect to the viscosity coefficient $\mu = 5 \times 10^{-3}, 1\times 10^{-2}, 5 \times 10^{-2}, 1 \times 10^{-1}$.
Here we have fixed $\rc = 1 \times 10^{-1}$.}
\label{fig:error-mu}
\end{figure}

We are particularly interested in the camphor motion $\xc$ and we have already given some interpretation of the short time approximate solution $\tilde{x}_c$ after \cref{thm:main}. In this section, we show the numerical results comparing these $\xc$ and $\tilde{x}_c$. In this section we
refer them as the solution and approximate solution, respectively.
All codes used for the numerical simulation are written by Julia~\cite{Bezanson2017}.

\subsection{Numerical methods}${}$
\par
To simulate the initial boundary value problem \eqref{eq:diffusion}--\eqref{eq:initial} numerically, 
we used the ADI (Alternating Direction Implicit) method (see~\cite{Douglas1964}, for example) for the diffusion equation~\eqref{eq:diffusion}
and the explicit Euler method for solving~\eqref{eq:ODE}.
The approximate solution was calculated by explicit Euler method
with the time integral approximated by the trapezoidal rule.
In both calculation, we approximated the line integral on the camphor boundary by the trapezoidal rule,
with interpolating the value of $u$ from the values on mesh by the bi-cubic method.

Take the domain $\Omega$ to be a square $[-8,8] \times [-8,8]$
and divide it into $800 \times 800$ meshes.
Then the length of the side of each mesh is $\Delta x = 2.0 \times 10^{-2}$.
Set the time step $\Delta t$ to be $\Delta t = 1.0 \times 10^{-4}$.

We calculated the error of the camphor position between the solution using the ADI method and approximate solution  by changing the viscosity $\mu$ and the camphor radius $\rc$.
The other parameters were fixed as follows: $D=1.0$, $\alpha=1.0$, $\beta_0 = 1.0$ and $\beta_1 = 0.5$.

Let the camphor source $f_{\rc}(x)=F \left( \dfrac{x}{\rc} \right)$
with $F(\xi) = 1$ if $\abs{\xi} < 1 - \delta$,
$F(\xi) = 0$ if $\abs{\xi} > 1 + \delta$
and interpolate it by a fourth order polynomial with respect to $\abs{\xi}$ in $1-\delta \le \abs{\xi} \le 1 + \delta$,
so that $F$ becomes $C^4$ function.
The initial value $u_0$ is set to $u_0(x_1,x_2) = 0.3 + 0.02(x_1+\sqrt{3}x_2)$.

\subsection{Simulation results}${}$
\par
\cref{fig:rc-error} shows that the normalized error $\rc^{-1}\abs{\xc-\tilde{x}_c}$ of camphor position
with respect to the normalized time $\rc^{-2}t$ by changing $\rc=0.2, 0.4, 0.5, 0.8, 1.0$.
Here we fixed $\mu = 1 \times 10^{-2}$ in this simulation.
From~\cref{fig:rc-error},
we see that the camphor position error decreases as the camphor radius $\rc$ becomes small.

The numerically computed respective orbits of the camphor are shown in~\cref{fig:orbits}.
Since our approximate solution is valid in short time, its orbit is getting off from that of the solution as time increases.

Next we compare the errors by changing the viscosity coefficient $\mu$ with fixed $\rc = 0.1$.
\cref{fig:error-mu} shows the error $\abs{\xc - \tilxc}$ of the camphor positions for the viscosity coefficients $\mu = 5 \times 10^{-3}, 1\times 10^{-2}, 5 \times 10^{-2}, 1 \times 10^{-1}$.
In~\eqref{eq:rel_dalta_B} of the proof of the error estimate,
as $\mu$ decreases, 
we need to take the length $\delta$ of the time interval smaller in order to obtain the same error.
In our simulation result, we see that the error decreases as $\mu$ increases.

\bigskip
\noindent
{\bf Acknowledgement}
\newline
The second author acknowledges the supports from Grant-in-Aid for Scientific Research (16H03949) of the Japan Scociety for the Promotion of Science (JSPS).
The third author acknowledges the supports from Grant-in-Aid for Scientific Research (15K21766 and 15H05740) of JSPS.

\bigskip
\bibliographystyle{siamplain}
\bibliography{ex_article}
\end{document}